\documentclass[11pt]{article}
	
	\newcommand{\blind}{0}
	
    \makeatletter
    \renewcommand\section{\@startsection {section}{1}{\z@}%
                                       {-3.5ex \@plus -1ex \@minus -.2ex}%
                                       {2.3ex \@plus.2ex}%
                                       {\normalfont\fontfamily{phv}\fontsize{16}{19}\bfseries}}
    \renewcommand\subsection{\@startsection{subsection}{2}{\z@}%
                                         {-3.25ex\@plus -1ex \@minus -.2ex}%
                                         {1.5ex \@plus .2ex}%
                                         {\normalfont\fontfamily{phv}\fontsize{14}{17}\bfseries}}
    \renewcommand\subsubsection{\@startsection{subsubsection}{3}{\z@}%
                                        {-3.25ex\@plus -1ex \@minus -.2ex}%
                                         {1.5ex \@plus .2ex}%
                                         {\normalfont\normalsize\fontfamily{phv}\fontsize{14}{17}\selectfont}}
    \makeatother

\usepackage{scalefnt}
\usepackage{amsmath}
\usepackage{amsthm}
\usepackage{comment}
\usepackage{booktabs}
\usepackage[dvipsnames]{xcolor}

\usepackage{natbib}
\usepackage{enumerate}
 \bibpunct[, ]{(}{)}{,}{a}{}{,}%

\usepackage{lscape}
\usepackage{graphicx}
\usepackage{hyperref}
\let\turkc\c
\usepackage{etoolbox}
\robustify\turkc

\newcommand{\C}{C}                      
\renewcommand{\c}{C'}                   
\renewcommand{\a}{a}               
\renewcommand{\P}{\gamma\cdot \eta}     
\renewcommand{\H}{P\cdot y_H}             
\renewcommand{\L}{P\cdot y_L}             

\usepackage{tikz}

\usepackage{chngpage}
\usepackage{mdframed}

\usepackage{caption}
\usepackage{subcaption}

\usepackage{natbib}
 \bibpunct[, ]{(}{)}{,}{a}{}{,}%

\usepackage{todonotes} 

\usepackage{lipsum}
\usepackage{appendix}
\usepackage{thm-restate}

\newtheorem{theorem}{Theorem}

\newtheorem{corollary}{Corollary}
\newtheorem{definition}{Definition}

\begin{document}

		\def\spacingset#1{\renewcommand{\baselinestretch}%
			{#1}\small\normalsize} \spacingset{1}
		
		\if0\blind
		{
			\title{\bf Analyzing Wage Theft in Day Labor Markets via Principal Agent Models}
			\author{James P. Bailey$^a$ and Bahar \turkc{C}avdar$^a$ and Yanling Chang$^b$ \\
			$^a$ \small{Department of Industrial and Systems Engineering,}\\ \small{Rensselaer Polytechnic Institute, Troy, NY} \\
             $^b$ \small{Department of Operations Research and Engineering Management},\\ \small{Southern Methodist University,  Dallas, TX}}
			\date{}
			\maketitle
		} \fi
		
		\if1\blind
		{

            \title{\bf \emph{IISE Transactions} \LaTeX \ Template}
			\author{Author information is purposely removed for double-blind review}
			
\bigskip
			\bigskip
			\bigskip
			\begin{center}
				{\LARGE\bf \emph{IISE Transactions} \LaTeX \ Template}
			\end{center}
			\medskip
		} \fi
		\bigskip
		
	\begin{abstract}
In day labor markets, workers are particularly vulnerable to wage theft. This paper introduces a principal-agent model to analyze the conditions required to mitigate wage theft through fines and establishes the necessary and sufficient conditions to reduce theft. We find that the fines necessary to eliminate theft are significantly larger than those imposed by current labor laws, making wage theft likely to persist under penalty-based methods alone. Through numerical analysis, we show how wage theft disproportionately affects workers with lower reservation utilities and observe that workers with similar reservation utilities experience comparable impacts, regardless of their skill levels. To address the limitations of penalty-based approaches, we extend the model to a dynamic game incorporating worker awareness. We prove that wage theft can be fully eliminated if workers accurately predict theft using historical data and employers follow optimal fixed wage strategy. 
Additionally, sharing wage theft information becomes an effective long-term solution when employers use any given fixed wage strategies, emphasizing the importance of raising worker awareness through various channels.

	\end{abstract}
			
	\noindent%

	\spacingset{1.5} 




\maketitle

\section{Introduction}\label{sec:Intro}
The day labor market in the United States has seen significant growth due to the complex interplay between labor supply and demand, transformative industrial shifts, and increased migration. Day laborers, who find substantial employment within the construction sector and related services, are sought after for their affordability, flexibility, and on-demand availability. 
It is estimated that  approximately 117,600 day laborers operate in the U.S. on any given day, with 49 percent hired by homeowners and renters, and 43 percent by construction contractors. 
Their primary occupations are construction laborers, gardeners, landscapers, painters, roofers, and drywall installers \citep{ruan2013, Valenzuela2006}.  
Every morning, these workers congregate at highly visible street corners, store parking lots, or worker advocacy centers waiting for employment opportunities. 
However, the absence of formal job protections exposes them to various risks including unsafe working conditions, severe injuries, and employer abuses. 
Among these, wage theft stands out as a particularly prevalent form of workplace exploitation.

Wage theft, commonly referred to as wage and hour violations in legal contexts, includes any unlawful actions that deprive employees of their rightful wages or benefits 
\citep{HernandezStepick2012, Dombrowski2017}.  These unlawful actions include misclassification of employees, minimum wage breaches, unpaid work hours, improper pay deductions, and denial of breaks \citep{Bobo2011, NCL2015, Robinson2011}.  A landmark survey of 4,387 low-wage workers in the three largest U.S. cities (New York City, Los Angeles, and Chicago) showed that more than 2/3 of workers suffered from at least one pay violation in the past workweek 
\citep{Bernhardt2009}, and the \textit{Employer Policy Foundation} estimated that workers would receive an additional 19 billion dollars annually if employers strictly followed wage and hour laws 
\citep{ruan2013}.
This problem has likely increased in recent years with a decrease in wage theft investigators; as of 2023, the Wage and Hour Division operates with a historical low of 733 investigators for the U.S. workforce, yielding approximately a ratio of one investigator for every 225,000 workers \citep{DOL2023}. 

In most cases, construction work is conducted through subcontracting, typically characterized by verbal agreements between subcontractors and workers, with wages frequently paid in cash. These practices create an unregulated environment where employers frequently fail to pay workers the owed wages \citep{Bobo2011, Bernhardt2009}. Furthermore, workers often refrain from reporting wage violations due to inadequate law enforcement,  restricted jurisdiction, and the low likelihood of inspections \citep{ruan2013}. 
In addition, many workers lack awareness of their rights or where to seek help for wage theft, and fear of immigration-related ramifications further discourages unauthorized migrants from reporting violations \citep{Theodore2016, Fussell2011, Lee2014}. These factors leave workers with limited recourse, forcing them to either silently accept wage theft, or endure lengthy delays in government response that may never come.

Several initiatives to address wage theft have been developed by both governmental and non-governmental organizations. For instance, Miami-Dade County in Florida introduced an ordinance to streamline the reporting and resolution of wage theft, creating a legal framework that enables workers to claim unpaid wages \citep{HernandezStepick2012}. Community-based organizations, often collaborating within wage theft task forces, contribute by advocating for workers' rights, providing educational resources, and raising public awareness about wage theft \citep{HernandezStepick2012}. Furthermore, the establishment of day labor and worker centers, such as Casa Latina in Seattle\footnote{https://casa-latina.org/work/day-worker-center/}, is crucial; 
these centers offer formal hiring processes that typically yield higher wages than informal sites \citep{Lee2014}. By negotiating better pay and promoting fair labor practices, these centers significantly enhance the income potential and workplace conditions for day laborers \citep{ Theodore2016, Melendez2014, kiley2024casa}. 

While these efforts provide valuable support to workers, current initiatives have not satisfactorily addressed wage theft issues, primarily due to limitations in enforcement capabilities, resource constraints, and the complex nature of the informal labor market. At the same time, academic research on this topic is still in its early stage, with sparse work seeking to gain a deeper understanding and develop more effective strategies to combat wage theft. Among those limited work,  \cite{CooperKroeger2017} conducted an empirical study using data from the \textit{Current Population Survey}\footnote{The Current Population Survey is the longest-standing government survey of labor market conditions in the United States and it is widely recognized as the primary public source of hourly wage data.}, highlighting the severity and prevalence of minimum wage violations -- a specific form of wage theft. Their results show  that (i) millions of workers are affected annually by this type of wage theft alone, and (ii) wage theft not only diminishes worker benefits but also significantly reduces government tax revenues from payroll and income taxes, effectively stealing from taxpayers who may end up supporting affected families through welfare programs. Analyzing the same data source, \cite{CLEMENS2022102285} showed that increases in the minimum wage are correlated with a rise in wage theft, accounting for 12 to 17 percent of the wage gains. \cite{KimAllmang2021} proposed a conceptual framework to analyze the primary causes of wage theft, attributing it to the low cost of enforcement and penalties due to outdated labor standards. They also reviewed policy measures aimed at updating these standards and improving enforcement. The seminal work of \cite{Weil2014} examined how changes in employment relationship, including subcontracting and outsourcing, have led to deteriorating labor conditions and increased instances of wage theft. \cite{Harkins2020} demonstrated that wage exploitation was a primary motive for the employment of migrant workers. \cite{DistelhorstMcGahan2022} used stakeholder theory and data from small manufacturers in emerging-market countries to show that abusive discipline and wage theft are linked to inferior firm-level manufacturing outcomes (such as late deliveries and low-quality products). Their analysis suggested that multinational corporations could enhance their corporate social performance by encouraging their suppliers to adopt value creation systems that prioritize the development of workers' human capital.

Despite the valuable insights provided by these empirical studies, theoretical analysis on wage theft and the broader issue of labor exploitation remain rare. Principal-agent models are particularly well-suited for analyzing these issues because they effectively capture the asymmetric information and conflicts of interest inherent in employer-employee relationships. Previous works, such as \cite{Chwe1990} and \cite{AcemogluWolitzky2011}, have utilized the principal-agent framework to explore coercion in labor markets. Similarly, \cite{BasuChau2004} applied this model to examine child labor and debt bondage, highlighting the dynamics between landlords and tenants in trapping children in labor. However, to our best knowledge, no theoretical work has specifically addressed wage theft. This paper fills this gap by employing a principal-agent framework to analyze wage theft, representing a first theoretical investigation of this issue.

In this paper, we consider a typical day labor market where the employer (the principal) offers a verbal contract specifying the wage paid to the worker (the agent) based on observed outcomes. The worker’s effort level probabilistically affects the output produced, and the worker decides on the level of effort to exert by balancing the disutility of effort against the expected wage. 
A third party, such as a non-profit organization, work center, or governmental agency, can inspect and penalize the employer if wage theft is detected. We explicitly model potential wage theft as part of the employer’s objective function. 
We consider two scenarios. In the first scenario, the worker is unaware of potential wage theft; while in the second scenario, the worker is aware of the employer’s wage theft history, possibly through social platforms or word-of-mouth, improving their situational awareness. We model the first scenario by a one-shot game, and the second by a repeated game. For both scenarios, we analyze the factors affecting wage theft and identify conditions under which wage theft can be prevented or reduced. The second scenario also sheds light on the value of improving worker situational awareness.

We formulate both scenarios as non-linear optimization problems using the principal-agent framework and provide a partial characterization of the optimal solution. Our theoretical results show that the problem can be reduced to a one-dimensional search, establishing a necessary foundation for analyzing the structure of the wage theft problem. This approach allows us to derive the following managerial insights. 
\begin{enumerate}[(i)]
    \item Establishing significant penalties is essential for effective wage theft mitigation. Our analysis indicates that even with frequent inspections, wage theft will persist if the penalties are insufficient (Theorem \ref{thm:eliminateTheft}). 
    This aligns with empirical evidence showing that strong penalties can effectively deter wage theft \citep{Galvin2016}.
    Currently, employers found guilty of underpaying employees often repay only a portion of the stolen wages and face few or low additional penalties. This insufficient penalization incentivizes employers to continue violating wage laws \citep{Galvin2016, kim2020wage}. 
    The necessary penalties we derive are drastically larger than those used in practice (e.g., \citealt{Backpay}), and therefore penalty-based methods alone are unlikely to eliminate wage theft.
Furthermore, our experiments reveal that increased inspections and penalties have no adverse effect on worker effort or project outcomes. These findings suggest that imposing stricter penalties and inspections not only curbs wage theft but may also enhance worker productivity, supporting the case for more robust enforcement of wage laws.

    \item Wage theft has a higher relative impact on workers with low reservation utilities, namely, workers with fewer and lower paying employment opportunities. Workers with lower reservation utilities are paid less and, as a result, have a larger portion of wages stolen. This finding may explain why particular groups such as migrants, undocumented workers, temporary laborers, and trafficked individuals, who commonly have lower reservation utilities, are reported to be more susceptible to wage theft by existing observations \citep{Bernhardt2009, Theodore2016}.

\item Workers with different skill levels experience similar amounts of wage theft when they have the same reservation utilities. Our computational results indicate that even if workers become more skilled, they will still face the same level of wage theft if their outside employment opportunities do not improve.
This is especially significant for vulnerable populations, where gaining new skills does not always lead to better job prospects or higher reservation utilities. However, it's important to note that our experiments isolate the impact of skill level, without considering changes in employment opportunities. For less vulnerable populations, improving skills often results in better job opportunities and higher reservation utilities, which can indirectly reduce the impact of wage theft.

        \item In addition to relying on external agency inspections and law enforcement to deter wage theft, we employ a repeated game model to theoretically examine the value of improving worker awareness, as some existing studies have suggested its potential benefits (e.g., \citealt{fine2010strengthening}). We prove that if the employer adopts optimal fixed wage strategies and the worker accurately predicts wage theft based on released historical data, wage theft can be fully eliminated. Furthermore, if the employer uses any given fixed wage strategies (not necessarily optimal), sharing information about wage theft becomes an effective long-term mitigation strategy, underscoring the importance of raising worker awareness through various channels such as social media, worker centers, community outreach programs, educational workshops, and collaboration with labor unions.
\end{enumerate}

Our paper is organized as follows. We state the wage theft problem in Section \ref{Sec:Model}. Section \ref{Sec:noAware} analyzes the case where the worker has no awareness of potential wage theft using a principal-agent model and presents a partial characterization of the optimal solution. Numerical analysis is performed in Section \ref{Sec:Numercial}, where we analyze the impact of penalty, reservation utility, and the worker's cost function on wage theft. In Section \ref{Sec:HasAwareness}, we theoretically examine the value of improving worker awareness on wage theft using repeated game theory. We conclude the paper in Section \ref{Sec:Conclusions} and provide directions for future research.

\section{Model Formulation}
\label{Sec:Model}

We model the wage exploitation environment through a principal-agent (employer-worker) model with the intervention of a third party (external agent), where the third part could be The Wage and Hour Division of the United States Department of Labor (WHD-DOL), local public security departments, non-profit anti-trafficking organizations, or worker centers. 
In this model, the employer takes over a one-day project and recruits a worker to engage in it. The project yields $y_H$ units of commodity products if successful and $y_L$ if unsuccessful, where $y_H>y_L \geq 0$.

At the recruitment stage, the employer specifies a \textit{verbal} output-dependent ``contract": the worker will receive a promised wage $w_i\geq 0$, $i \in \{H,L\}$ for high $(y_H)$ and low $(y_L)$ output, respectively. 
The worker may choose to either accept or reject the contract. 
If the offer is rejected, the worker receives payoff equal to his/her outside option (or reservation utility) $u$, and the employer receives zero payoff. 
If the contract offer is accepted, then the worker decides to exert effort level $a \in [0,1)$ at cost $\C(a)$, where $a$ is also the probability of obtaining output $y_H$. 
Following the common assumptions on the cost of worker's efforts in the labor economy literature, we assume that $\C(0)=0$, and function $\C(\cdot)$ is increasing, strictly convex, twice differentiable, and $\lim_{a \rightarrow 1} \C(a) = \infty$ \citep{AcemogluWolitzky2011}. 
We remark that this also implies $\c(a)\to \infty$ as $a\to 1$; if $\c(a)$ is bounded by $w$, then $\C(a)=\int_0^a \c(x)dx\leq w\cdot a\leq w$ is also bounded. 

We denote the market price for the unit product by $P$ and assume that the employer cannot alter the product price. 
At the end of the working day, the employer may act against the verbal agreement and withhold $b_i$ amount of the promised wage for $i\in  \{H,L\}$. 
If $b_i>0$, an incidence of wage theft occurs. 
We denote the probability that the external agent inspects the employer by $\gamma$. We assume that each inspection can detect the wage theft if it exists, so we will use the inspection and detection rates interchangeably. 
We model the resulting penalty through a function $\eta(b_i)$,  where $\eta(\cdot )$ is strictly convex, increasing, and twice differentiable with $\eta(0)=0$.

Considering this setting, we study the wage theft problem under two scenarios. In the first scenario, we assume that the worker has no awareness of the wage theft; in the second scenario, we consider cases where the worker has information about the wage theft. 

\section{Worker with No Prior Knowledge of Wage Theft}
\label{Sec:noAware}
When the worker has no awareness of potential wage theft, the worker develops a strategy based on the verbal contract only, as in the base principal-agent problem. 
However, the employer can offer a different contract by withholding some of the wage at the end of the day with the risk of financial consequences if caught by the inspector. 
Given the sequence of events in the game, the optimal verbal contract is a solution to the following optimization problem:
\begin{align}
    \max_{\a, w_H, w_L, b_H, b_L} &\ \a\cdot (\H - w_H +b_H - \P(b_H))+(1-\a)\cdot (\L - w_L +b_L - \P(b_L)) \label{Eq:Obj}\tag{Employer Profit}\\
  \a &\in \arg\max_{a'\in [0,1)} \left\{ \a'\cdot w_H + (1-\a')\cdot w_L - \C(\a') \right\}\label{Eq:IC}\tag{Incentive Compatibility}\\
   \a&\cdot w_H + (1-\a)\cdot w_L - \C(a) \geq u\label{Eq:IR}\tag{Individual Rationality}\\  
    0&\leq b_L \leq w_L, \ \
    0\leq b_H \leq w_H, \label{Eq:theft}\tag{Wage Bounds}\\
    \a& \in [0,1) \tag{Probability of High Outcome}
    \label{Eq:Bounds}
\end{align}

The (\ref{Eq:Obj}) objective function calculates the principal's expected profit considering the verbal contract, the best response of the worker, and the gain and cost of wage theft. 
The (\ref{Eq:IC}) constraint determines the worker's best response to the verbal contract. 
The (\ref{Eq:IR}) constraint ensures the agent will only participate if accepting the offer is better than their reservation utility (outside option). 
Constraints (\ref{Eq:theft}) state that the wage theft amounts are limited by the verbally offered wages in the contract. 
The  (\ref{Eq:Bounds}) constraints provide the bounds for the agent's effort.

In the following section, we explore the structural properties of the optimal solution for the principal's problem to determine the verbal contract under wage theft.
We begin by providing a partial characterization for this principal-agent model. 
Specifically, we determine the optimal wage and wage thefts needed to induce a fixed given effort level $a$, i.e., $w_i^*(a)$ and $b_i^*(a)$ as functions of the effort level $a$, for $i\in \{H,L\}$.

\subsection{Mitigating Theft via Fines}\label{sec:Fines}

We begin by determining the optimal theft by the employer as a function of the effort level and offered wages.

\begin{definition}[Principal's ideal wage theft]
\label{def:Ideal}
We call the wage theft amount $\beta$ that satisfies $ \P'(\beta)=1$ the principal's ideal wage theft.
\end{definition}

The principal's ideal wage theft is the amount the principal would steal from the worker if there were no penalties; it represents where the marginal benefit of theft matches the marginal cost due to the penalty of theft. 
Using this definition, we present initial results, which also lead to reducing the original optimization problem to a problem involving only a single decision variable.

\begin{restatable}[Optimal Theft]{proposition}{IdealTheft}
\label{prop:IdealTheft}
Let $w_i^*(a)$ be an optimal promised wage that induces arbitrary (not necessarily optimal) effort level $a\in [0,1)$. 
Then a corresponding best wage theft amount for the employer is $b_i^*(a)=\min\{\beta, w_i^*(a)\}$ for $i\in \{L,H\}$. 
Furthermore, if $a\in (0,1)$, then $b_i^*(a)=\min\{\beta, w_i^*(a)\}$ is the unique optimal wage theft. 
\end{restatable}

The proof of Proposition \ref{prop:IdealTheft} appears in Appendix \ref{app:theft}.
The employer will try to withhold as much as $\beta$ from the promised wage. 
If $\beta\geq w_i^*$, then the employer withholds the whole wage, and the worker receives nothing. 
If $\beta<w_i^*$, then the employer steals the amount $\beta$, and the worker obtains $w_i^*-\beta$. 
In the latter case, the employer does not intend to withhold the whole wage due to the expected penalty cost. By altering $\beta$, Proposition \ref{prop:IdealTheft} supports our intuition that the inspector can reduce the wage theft by either increasing the frequency of inspections (i.e., larger $\gamma$) or by penalizing the wage theft more aggressively (i.e., a steeper $\eta(\cdot)$ function).
Further, Proposition \ref{prop:IdealTheft} immediately provides necessary and sufficient conditions to eliminate theft via imposed penalties. 

\begin{theorem}\label{thm:eliminateTheft}
    To eliminate wage theft, i.e., to force $\beta=0$, the inspection rate $\gamma$ and penalty function $\eta(\cdot)$ must satisfy  $\eta'(0)\geq 1/\gamma$.
\end{theorem}

While this result seems promising for eliminating wage theft, this requirement is unlikely to be realistic given current practices of inspection rates and penalties. According to \cite{Backpay}, most cases of wage theft go undetected  ($\gamma$ is close to zero) and the maximum penalty is capped at only twice the amount of theft (employers can be required to pay back wages plus an equal amount in liquidated damages). 
For a penalty function describing this setting where the penalty is bounded by twice the wage theft, it is straightforward to show that $\eta'(0)\leq 2$, and therefore the inspection rate would have to be at least $\gamma = 0.5$ to eliminate wage theft. 
In other words, when using only a penalty-based method, wage theft is likely to still exist even if the inspection rate is very high since current penalties tend to be low.
For the remainder of this section, we focus on partially characterizing the optimal solution in order to understand how wage theft will impact different types of workers. 

\subsection{Optimal Wages to Induce A Specific Effort Level}

In this subsection, we focus on determining the optimal wages, $w_H^*(\a)$ and $w_L^*(\a)$ needed to induce a specific effort level $\a$.

\begin{restatable}[Optimal Wages]{proposition}{Wages}
\label{prop:Wages}
An optimal set of promised wages to induce arbitrary (not necessarily optimal) effort level $\a$ is:
\begin{align*}
    w_L^*(a)&=\max\{0, u-a\cdot \c(a)+\C(a)\} \tag{Optimal Low Wage to Induce Effort $\a$}\\
    w_H^*(a)&=w_L^*(a)+\c(a)=\max\{\c(\a), u+(1-a)\cdot \c(a)+\C(a)\}
    \tag{Optimal High Wage to Induce Effort $\a$}\\
\end{align*}
\vspace{-.6in}

\noindent Further, if $\a\in (0,1)$ then $w_L^*(\a)$ and $w_H^*(a)$ are the unique promised wages to induce $\a$. 
\end{restatable}

The proof of Proposition \ref{prop:Wages} appears in Appendix \ref{app:wages}.
Combining Propositions \ref{prop:IdealTheft} and \ref{prop:Wages}, an optimal set of wages and wage thefts to induce an arbitrary effort level $a$ can be expressed as follows, which reduces the principal-agent model to a single decision variable.
\begin{align*}
    {w_L^*(\a)}&=\max\{0, u-\a\cdot \c (\a)+\C (\a)\} \tag{low wage}
    \\
    w_H^*(\a)&=w_L^*(\a)+\c (\a)=\max\{\c(\a), u + (1-\a)\cdot \c(\a) + \C(\a)\}\tag{high wage}\\
    b_L^*(\a)&=\min\{ \beta, w_L^*(\a))=\min\{ \beta, \max\{0, u-\a\cdot \c (\a)+\C (\a)\}\} \tag{low theft}\\
    b_H^*(\a)&=\min\{ \beta, w_H^*(\a))=\min\{ \beta, \max\{0, u-\a\cdot \c (\a)+\C (\a)\}+\c (\a)\} \tag{high theft}
    \end{align*}

Based on these results, we see that there are several key factors playing a role on the wage theft; not only does the penalty function play a significant role, both the reservation utility $u$ and worker cost $\C$ impact the realized wage theft. 
To better understand the dynamics of wage theft, we present an example below.

\subsubsection*{An Illustrative Example}
We demonstrate the results in Propositions \ref{prop:IdealTheft} and \ref{prop:Wages} on an example where the worker's cost function $\C(\a)=\a / (1-\a)$ and outside option $u=1$ in Figures \ref{fig:wHwL} and \ref{fig:4Cases}. 
Figure \ref{fig:wHwL} presents the wages in the verbal contract, and Figure \ref{fig:4Cases} presents the corresponding wage theft amounts in detail.
In Figure \ref{fig:wHwL}, we see that in order to induce higher effort, the offered high wage needs to increase while the low wage offered generally decreases.
This is intuitive as the worker's marginal revenue for additional effort is precisely the difference between the high and low wages. 
Moreover, it is also straightforward to show theoretically since $w_H^*(a)$ and $w_L^*(a)$ are increasing and non-decreasing functions of $a$ respectively. 
As shown in the figure, this means there exists a discontinuity in $w_H^*(a)$ and $w_L^*(a)$ for $a$ where $u=\a \C^\prime (\a) + \C(\a)$. 
Below this threshold represents situations where the probability of a low outcome is high enough that the employer must promise positive low wages as insurance to ensure the expected outcome for a worker exceeds their reservation utility.

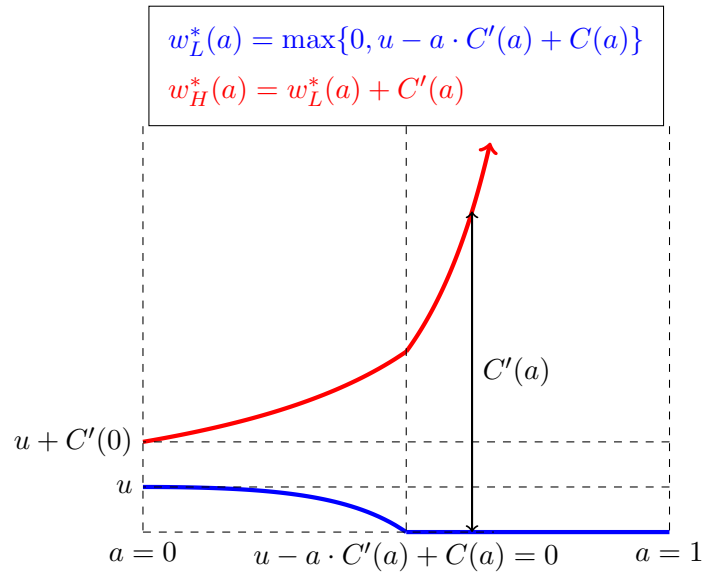
\begin{figure}[ht]
    \centering
    \begin{tikzpicture}[xscale=7,yscale=0.6]
    
      \draw[domain=0:0.5, smooth, ultra thick, variable=\x, blue] plot ({\x}, {1- \x*\x/(1-\x)/(1-\x)});
      \draw[domain=0.5:1, smooth, ultra thick, variable=\x, blue] plot ({\x}, 0);
      
      \draw[domain=0:0.5, smooth, ultra thick, variable=\x, red] plot ({\x}, {1+ (1-\x*\x)/(1-\x)/(1-\x)});
      \draw[domain=0.5:.66, smooth, ultra thick, variable=\x, red,->] plot ({\x}, {1/(1-\x)/(1-\x)});
      
      
        \draw[dashed] (0.5,9)--(0.5,0) node[below] {$u-\a\cdot \c (\a) + \C (\a)=0$};
        \draw[dashed] (1,9)--(1,0) node[below] {$\a=1$};
        \draw[dashed] (0,9)--(0,0) node[below] {$\a=0$};
      
        \draw[dashed] (0,0)--(2/3,0);
        \draw[dashed] (1,1)--(0,1) node[left] {$u$};
        \draw[dashed] (1,2)--(0,2) node[left] {$u+ \c(0)$};
      
        \draw[thick, <->] (0.625,0) -- (0.625,7.111) node[midway, right] {$\c (\a)$};

      

    \matrix [draw, above ] at (0.5,9) {
        \node [blue,right] {$w_L^*(\a)= \max \{ 0, u- \a \cdot \c (\a) + \C (\a)\}$}; \\
        \node [red,right] {$w_H^*(\a)=w_L^*(\a) + \c (\a)$}; \\
    };
      
    \end{tikzpicture}
    \caption{Optimal values for $w_H$ and $w_L$ as functions of effort $\a$ (Proposition \ref{prop:Wages}) for cost function $\C(\a)=\a/(1-\a)$ and outside option $u=1$.} \label{fig:wHwL}
\end{figure}

Figure \ref{fig:4Cases} depicts the optimal wage theft amounts alongside the corresponding optimal wages to induce a specific effort level. 
The impact of wage theft on workers varies based on the severity of wage theft penalties (by the amount of the ideal theft $\beta$). 
When the penalties for wage theft are sufficiently large, e.g., either the inspection is frequent or the penalty is large, then the ideal theft falls in the first region where $\beta \in (0, u]$ (Figure \ref{subfig:4a}). 
In this case, for all effort levels, workers with high output will be subject to the same wage theft amount with positive net pay since the promised wages exceed the ideal theft $\beta$.
However, when low output occurs, worker's will receive an \textit{effective wage} (i.e., verbal wage minus the wage theft, $w_i-b_i$ for $i\in\{H, L\}$) of zero if the effort level is sufficiently high, even if they were promised positive wages.
At the other extreme, when the penalty for theft is sufficiently small, i.e., $\beta \geq u$, the wage theft will always exceed the promised low wages resulting in no payment when a low outcome is observed regardless of whether wages were promised as depicted in Figures \ref{subfig:4b} and \ref{subfig:4c}.

\definecolor{MyColor}{RGB}{218,170,32}

\begin{figure}[!ht]
    \centering
    \begin{subfigure}[t]{.32\textwidth}
    \begin{tikzpicture}[xscale=5,yscale=0.4]

      \draw[domain=0:0.5, smooth, ultra thick, variable=\x, blue] plot ({\x}, {1- \x*\x/(1-\x)/(1-\x)});
      \draw[domain=0.5:1, smooth, ultra thick, variable=\x, blue] plot ({\x}, 0);

      \draw[domain=0.414:0.5, smooth, line width=1.01mm, Magenta, dashed, variable=\x] plot ({\x}, {1- \x*\x/(1-\x)/(1-\x)});
      \draw[line width=1.01mm, Magenta] (0,0.5)--(.414,0.5);
      \draw[line width=1.01mm, Magenta, loosely dashed] (0.5,0)--(1,0);

      \draw[domain=0:0.5, smooth, ultra thick, variable=\x, red] plot ({\x}, {1+ (1-\x*\x)/(1-\x)/(1-\x)});
      \draw[domain=0.5:.66, smooth, ultra thick, variable=\x, red,->] plot ({\x}, {1/(1-\x)/(1-\x)});

      \draw[line width=1.01mm, MyColor, dashed] (0,0.5)--(1,0.5);
      
      
        \draw[dashed] (1,9)--(1,0) node[below left] {$\a=1$};
        \draw[dashed] (0,9)--(0,0) node[below right] {$\a=0$};
        \draw[dotted] (0.5,9)--(0.5,0);
        \draw[dotted] (0.414,9)--(0.414,0);
      
        \draw[dashed] (0,0)--(2/3,0);
        \draw[black, line width=.2mm] (1,.5)node[above left] {$\beta : \P'(\beta)=1$} --(0,.5) ;

      \fill[white] (0.0125, 9.3)--(0.0125, 6.9)--(0.595,6.9)--(0.595,9.3)--cycle;
      \draw[thick] (0.0125, 9.3)--(0.0125, 6.9)--(0.595,6.9)--(0.595,9.3)--cycle;
      
      \draw[thick, blue] (0.025,8.7)--(0.05,8.7) node[right, blue] {$w_L^*(\a)$};
      \draw[thick, red] (0.025,7.5)--(0.05,7.5) node[right, red] {$w_H^*(\a)$};
      
      \draw[line width=1.01mm, Magenta, dashed] (0.3,8.7)--(0.35,8.7) node[right, Magenta] {$b_L^*(\a)$};
      \draw[line width=1.01mm, MyColor, dashed] (0.3,7.5)--(0.35,7.5) node[right, MyColor] {$b_H^*(\a)$};
      
    \end{tikzpicture}
    \caption{$\beta \in (0,u]$}\label{subfig:4a}
    \end{subfigure}
    \hfill
    \begin{subfigure}[t]{.32\textwidth}
    \begin{tikzpicture}[xscale=5,yscale=0.4]

      \draw[domain=0:0.5, smooth, ultra thick, variable=\x, blue] plot ({\x}, {1- \x*\x/(1-\x)/(1-\x)});
      \draw[domain=0.5:1, smooth, ultra thick, variable=\x, blue] plot ({\x}, 0);

      \draw[domain=0:0.5, smooth, line width=1.01mm, Magenta, dashed, variable=\x] plot ({\x}, {1- \x*\x/(1-\x)/(1-\x)});
       \draw[line width=1.01mm, Magenta, loosely dashed] (0.5,0)--(1,0);

      \draw[domain=0:0.5, smooth, ultra thick, variable=\x, red] plot ({\x}, {1+ (1-\x*\x)/(1-\x)/(1-\x)});
      \draw[domain=0.5:.66, smooth, ultra thick, variable=\x, red,->] plot ({\x}, {1/(1-\x)/(1-\x)});

      \draw[line width=1.01mm, MyColor, dashed] (0,1.5)--(1,1.5);
      
      
        \draw[dashed] (1,9)--(1,0) node[below left] {$\a=1$};
        \draw[dashed] (0,9)--(0,0) node[below right] {$\a=0$};
        \draw[dotted] (0.5,9)--(0.5,0);
      
        \draw[dashed] (0,0)--(2/3,0);
        \draw[black, line width=.2mm] (1,1.5)node[above left] {$\beta : \P'(\beta)=1$} --(0,1.5) ;

      \fill[white] (0.0125, 9.3)--(0.0125, 6.9)--(0.595,6.9)--(0.595,9.3)--cycle;
      \draw[thick] (0.0125, 9.3)--(0.0125, 6.9)--(0.595,6.9)--(0.595,9.3)--cycle;
      
      \draw[thick, blue] (0.025,8.7)--(0.05,8.7) node[right, blue] {$w_L^*(\a)$};
      \draw[thick, red] (0.025,7.5)--(0.05,7.5) node[right, red] {$w_H^*(\a)$};
      
      \draw[line width=1.01mm, Magenta, dashed] (0.3,8.7)--(0.35,8.7) node[right, Magenta] {$b_L^*(\a)$};
      \draw[line width=1.01mm, MyColor, dashed] (0.3,7.5)--(0.35,7.5) node[right, MyColor] {$b_H^*(\a)$};
      
    \end{tikzpicture}
    \caption{$\beta \in [u,u+\c(0)]$}\label{subfig:4b}
    \end{subfigure}
    \hfill
    \begin{subfigure}[t]{.32\textwidth}
    \begin{tikzpicture}[xscale=5,yscale=0.4]

      \draw[domain=0:0.5, smooth, ultra thick, variable=\x, blue] plot ({\x}, {1- \x*\x/(1-\x)/(1-\x)});
      \draw[domain=0.5:1, smooth, ultra thick, variable=\x, blue] plot ({\x}, 0);

      \draw[domain=0:0.5, smooth, line width=1.01mm, Magenta, dashed, variable=\x] plot ({\x}, {1- \x*\x/(1-\x)/(1-\x)});
      \draw[line width=1.01mm, Magenta, loosely dashed] (0.5,0)--(1,0);

      \draw[domain=0:0.5, smooth, ultra thick, variable=\x, red] plot ({\x}, {1+ (1-\x*\x)/(1-\x)/(1-\x)});
      \draw[domain=0.5:.66, smooth, ultra thick, variable=\x, red,->] plot ({\x}, {1/(1-\x)/(1-\x)});

      \draw[domain=0:0.333, smooth,line width=1.01mm, MyColor, dashed,variable=\x] plot ({\x}, {1+ (1-\x*\x)/(1-\x)/(1-\x)});
      \draw[line width=1.01mm, MyColor, dashed] (0.333,3)--(1,3);
      
      
        \draw[dashed] (1,9)--(1,0) node[below left] {$\a=1$};
        \draw[dashed] (0,9)--(0,0) node[below right] {$\a=0$};
        \draw[dotted] (0.5,9)--(0.5,0);
        \draw[dotted] (0.333,9)--(0.333,0);
      
        \draw[dashed] (0,0)--(2/3,0);
        \draw[black, line width=.2mm] (1,3)node[below left] {$\beta : \P'(\beta)=1$} --(0,3) ;

      \fill[white] (0.0125, 9.3)--(0.0125, 6.9)--(0.595,6.9)--(0.595,9.3)--cycle;
      \draw[thick] (0.0125, 9.3)--(0.0125, 6.9)--(0.595,6.9)--(0.595,9.3)--cycle;
      
      \draw[thick, blue] (0.025,8.7)--(0.05,8.7) node[right, blue] {$w_L^*(\a)$};
      \draw[thick, red] (0.025,7.5)--(0.05,7.5) node[right, red] {$w_H^*(\a)$};
      
      \draw[line width=1.01mm, Magenta, dashed] (0.3,8.7)--(0.35,8.7) node[right, Magenta] {$b_L^*(\a)$};
      \draw[line width=1.01mm, MyColor, dashed] (0.3,7.5)--(0.35,7.5) node[right, MyColor] {$b_H^*(\a)$};

    \end{tikzpicture}
    \caption{$\beta \in [u+\c(0), \infty$)}\label{subfig:4c}
    \end{subfigure}       
    \caption{Optimal wage theft amounts $b_H$ and $b_L$ as functions of effort $\a$ and ideal theft $\beta$ for cost function $\C(\a)=\a/(1-\a)$ and outside option $u=1$. The ideal wage theft amount $\beta$ increases from (a) to (c), implying that the wage theft becomes more costly for the employer. The vertical lines indicate places where the first derivative is discontinuous.} \label{fig:4Cases}
\end{figure}
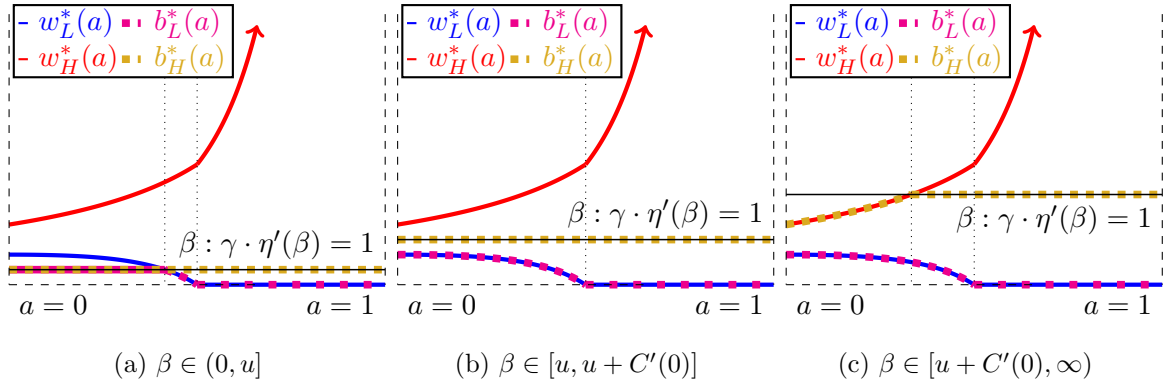


\subsection{On the Optimal Effort Level}\label{sec:OptEffort}

Propositions \ref{prop:IdealTheft} and \ref{prop:Wages} show that all $w_i^*$ and $b_i^*$ values are functions of the worker's effort level $\a$, $i \in \{L,H\}$. This means that the original optimization problem can be represented with the one-dimensional search problem below:
\begin{align*}
    \max_{\a\in [0,1)} g(\a) \tag{One-dimensional search problem}\label{eqn:1DSearch}
\end{align*}
\begin{align*}
    g(\a):= &\ \a\cdot \left( \H - w_H^*(\a) + b_H^*(\a) - \P(b_H^*(\a))\right)\\ &+ (1-\a)\cdot \left( \L - w_L^*(\a) + b_L^*(\a) - \P(b_L^*(\a))\right),
\end{align*}
where $w_i^*(\a)$ and $b_i^*(\a)$ for $i\in \{H,L\}$ are given in Propositions \ref{prop:IdealTheft} and \ref{prop:Wages}.

The one-dimensional search problem is defined in an open space $[0,1)$, and it is not immediately clear that an optimal solution exists. 
Proposition \ref{prop:existence} proves the existence of an optimal solution for this problem. 

\begin{restatable}{proposition}{Existence}\label{prop:existence}
    An optimal solution to the (\ref{eqn:1DSearch}) exists with $\a^*\in [0,1)$. 
\end{restatable}

The proof of Proposition \ref{prop:existence} appears in Appendix \ref{app:exist}.
Proposition \ref{prop:existence} alone does not exclude the possibility that the agent exerts no effort, which is considered necessary in the labor economics literature. 
The next proposition presents a necessary and sufficient condition for $a^* > 0$, which states that the worker will exert positive effort as long as the potential production gain exceeds the worker's marginal cost of effort at $\a=0$. 
The condition is relatively natural; it simply states the marginal gains to yield a high outcome must exceed the marginal cost for a worker to exert effort. 

\begin{restatable}{proposition}{Positive}\label{prop:positive}
    A sufficient condition to ensure $g'(0)>0$ (thereby ensuring $a^*>0$) is $\c(0)<\H-\L$.  Further, this condition is necessary when $0\leq \beta\leq u$.
\end{restatable}

The proof of Proposition \ref{prop:positive} appears in Appendix \ref{app:positive}.
Notably, Propositions \ref{prop:existence} and \ref{prop:positive} imply $a^*\in (0,1)$ thereby strengthening Propositions \ref{prop:IdealTheft} and \ref{prop:Wages}.

\begin{corollary}\label{cor:uniqueness}
    Suppose $\c(0)<\H-\L$ and $a^*$ is an optimal-effort level.  
    Then $b_i^*(a^*)$ and $w_i^*(a^*)$ for $i\in \{L,H\}$ as given in Propositions \ref{prop:IdealTheft} and \ref{prop:Wages} are the unique optimal wage theft and wages respectively. 
\end{corollary}

Finally, we remark that (\ref{eqn:1DSearch}) has a discontinuous first derivative, making it difficult to provide a complete characterization of the optimal solution. 
In the following section, we aim to understand the impact of wage theft on different types of workers by relying on a numerical analysis in conjunction with our partial characterization.

\section{Numerical Analysis}
\label{Sec:Numercial}

We have thus far analyzed the wage theft decisions and their impact on workers in a general theoretical context, without requiring specific functional forms for cost or penalty. In this section, we present numerical results using particular functional forms for the worker's cost of effort and the employer's penalty, adhering to the general characteristics of these functions. This allows us to offer concrete observations to complement our theoretical findings and provide further insights.

We found challenges in collecting wage theft related data since the incidences are not regularly reported (with the exception of certain data, such as the Current Population Survey on minimum wage violations, which represents just one type of wage theft), and the publicly available data is known to be quite limited \citep{CooperKroeger2017}. Therefore, we perform numerical analysis on general function forms with a set of parameter settings to obtain more comprehensive observations.
For the worker's cost function, we use $C(a)= \frac{ka}{(1-a)^q}$, where $k$ and $q$ are positive constants. Note that this function maintains all properties of the cost function as described in Section \ref{Sec:Model}.
We model the penalty function as $\eta(b)=\sigma b ^p$, where $\sigma$ is the penalty coefficient and $p$ is the growth factor, and they are both positive.
By changing these values, we can represent various penalty scenarios. 
 The parameter values used in the numerical study are presented in Table \ref{Tab:Param}.

\begin{table}[!ht]
\centering
\caption{Parameter values in the numerical study.}
\begin{tabular}{@{} ll @{}}    \toprule
\emph{Parameter} & \emph{Values}   \\\midrule
Price $P$    & 10, 15, 20, 30, 40   \\ 
Low output $y_L$ & 30, 40 \\
High output $y_H$ & 50\\
Penalty coefficient $\sigma$ & 0.25, 0.5, 0.75, 1, 1.25, 1.50, 2, 3, 4, 5 \\
Penalty growth factor $p$ & 1.1, 1.2, 1.3, 1.4, 1.5 \\
Inspection rate $\gamma$ & 0.1, 0.2, 0.3, 0.4, 0.5, 1 \\
Worker's cost coefficient $k$ & 0.1, 0.5, 1\\
Worker's cost growth factor  $q$ &0.1, 0.5, 1, 3, 5\\
Worker's external utility $u$ &10, 25, 50, 100, 200, 300, 400, 500, 600\\
 \bottomrule
\end{tabular}
\label{Tab:Param}
\end{table}

 In our discussion below, we focus on a subset of parameters with the biggest impact on wage theft to better understand how to reduce or eliminate the wage theft. We present the computational results via plotting the optimal verbal wages, wage thefts and the optimal effort levels. In each plot, the left $y$-axis shows the wages and right $y$-axis shows the effort level.

\textbf{Impact of penalty policy on wage theft:} To decrease the wage theft, the external agent (e.g., The Department of Labor) controls two main factors: the frequency of inspection, which determines the detection rate $\gamma$, and the penalty function $\eta(b)$.
Recall from Section \ref{sec:Fines} that the principal's ideal wage theft is $\beta$ where $\eta'(\beta)=1/\gamma$.  
For the penalty function $\eta(b)=\sigma b^p$, the principal's ideal wage theft is $\beta=(p\sigma\gamma)^{1/(1-p)}$.

In Figure \ref{fig:sigmaEffect}, we present the outcomes of the wage theft problem with respect to increasing $\sigma$ values, respectively for low ($\gamma=0.2$) and high ($\gamma=0.5$) detection rates. 
As expected from our theoretical analysis on $\beta$, we see that as the penalty coefficient $\sigma$ increases the wage theft goes to zero, making the effective wages $w_i-b_i$ equal to the promised verbal wages.
Notably, this shift occurs more rapidly with higher inspection rates (Figure \ref{fig:sub11} versus  \ref{fig:sub12}), i.e., both inspection rates and penalties are important for mitigating wage theft. 

\begin{figure}[!ht]
\centering
\begin{subfigure}{0.5\textwidth}
  \centering
  \includegraphics[width=1\linewidth]{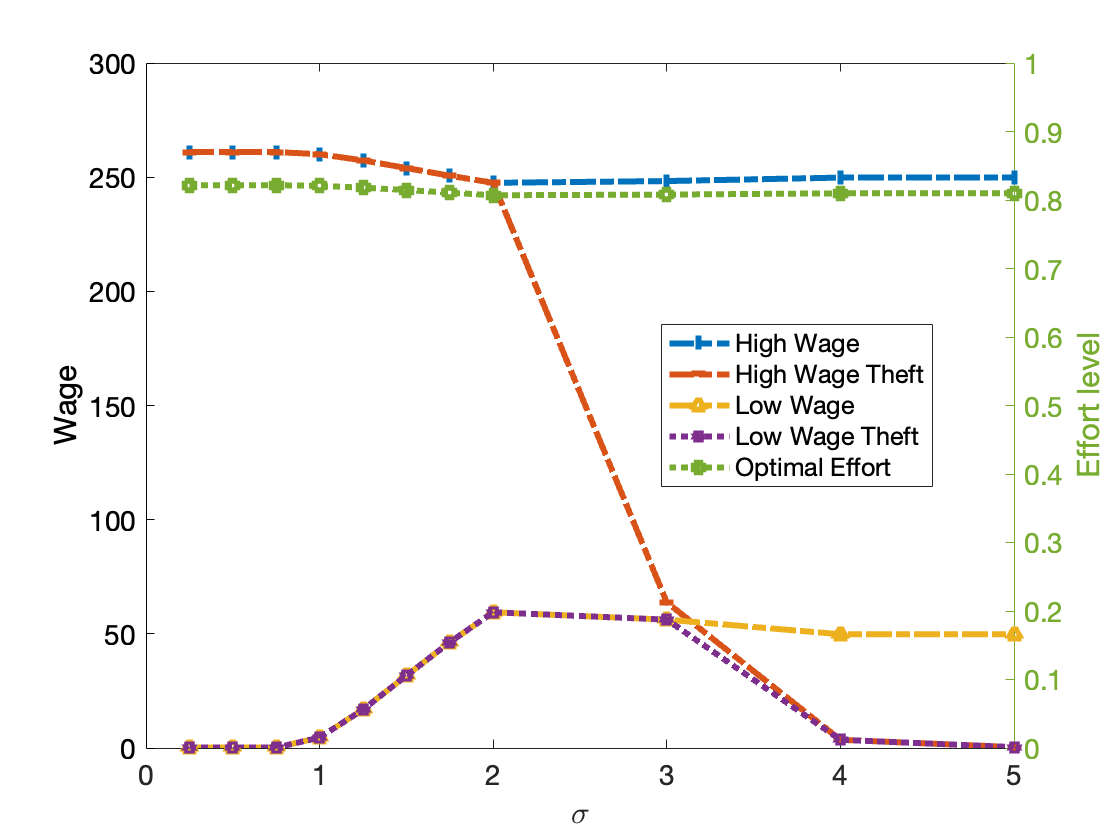}
  \caption{Low Inspection Rate ($\gamma =0.2$).}
  \label{fig:sub11}  
\end{subfigure}%
\begin{subfigure}{0.5\textwidth}
  \centering
  \includegraphics[width=1\linewidth]{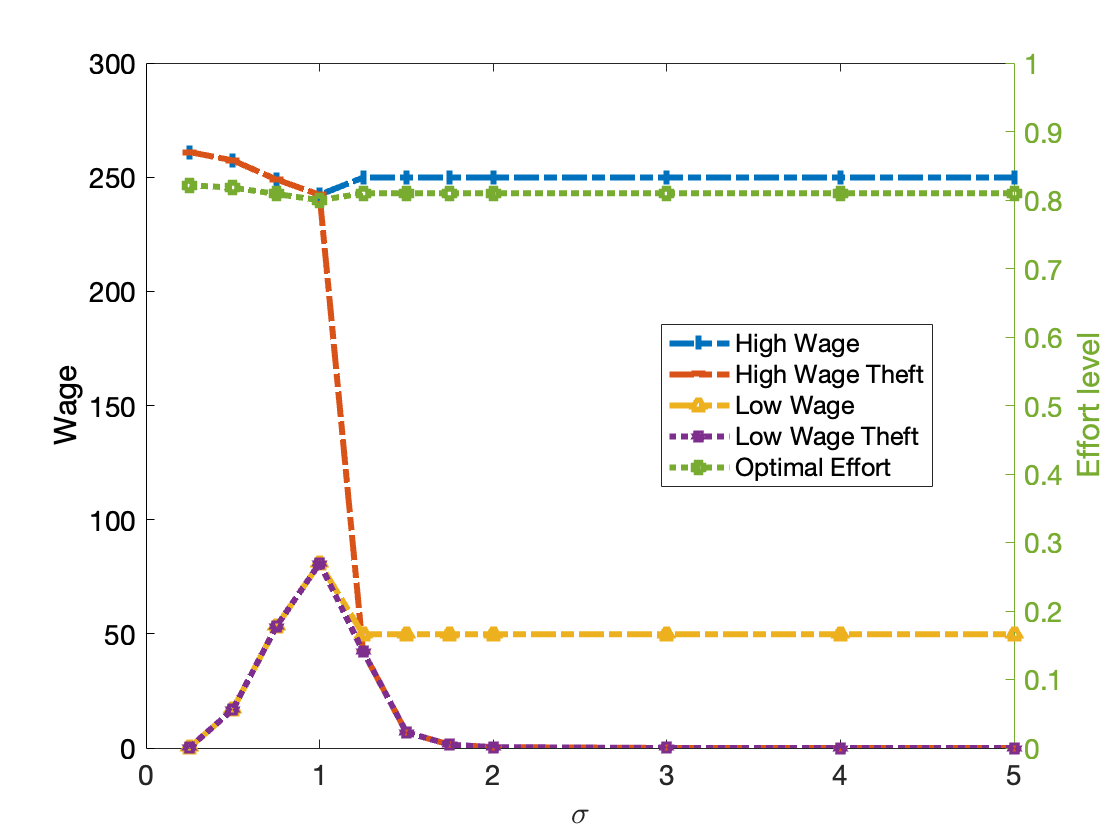}
  \caption{High Inspection Rate ($\gamma =0.5$).}
  \label{fig:sub12}
\end{subfigure}
\label{fig:sigmaEffect}
\caption{Impact of $\sigma$ for different inspection rates ($P=10$, $y_H=50$, $y_L=30$, $u=200$, $k=0.1$,  $q=3$, $p=1.1$).}
\label{fig:sigmaEffect}
\end{figure}

Another observation from Figure \ref{fig:sub11} is that when the penalties for theft are large (e.g., $\sigma>3$) that the promised wages are relatively stable and the amount of wage theft decays to zero. 
However, in contrast, when the barrier to theft is low (e.g., $\sigma <1$), the employer is able to artificially inflate the promised high wage to induce slightly higher effort with no intention of paying the worker. 
Further, since high wages are inflated, the employer is also able to promise the worker lower low wages to further disincentivize low level effort.
As the barrier to theft increases ($\sigma \in (1,2)$), we observe that the employer can no longer inflate promised wages. 
As a result, to ensure the worker's utility is at least the reservation utility (\ref{Eq:IR}), the employer must also increase the promised low wage.

We also observe that the penalty coefficient $\sigma$ has little impact on the optimal effort level. 
Our numerical experiments suggest that increasing penalties can effectively reduce wage theft without negatively impacting the output of the workers.

We also demonstrate the impact of inspection frequency in Figure \ref{fig:gammaEffect} for different penalty coefficients, $\sigma\in \{0.5, 1\}$. 
Similar to the impact of $\sigma$, we see that the wage theft amounts decrease as the external agent places more efforts into inspecting the employer, and the effective wages become closer to the promised wages. The required frequency of inspection to reduce the wage theft to a certain level depends on the penalty coefficient. For example, as seen in Figures \ref{fig:sub31} and \ref{fig:sub32}, similar wage theft quantities occur under high inspection ($\gamma=1$) with small penalty ($\sigma=0.5$) and under medium inspection ($\gamma=0.5)$ with high penalty ($\sigma=1$).

\begin{figure}[ht]
\centering
\begin{subfigure}{0.5\textwidth}
  \centering
  \includegraphics[width=1\linewidth]{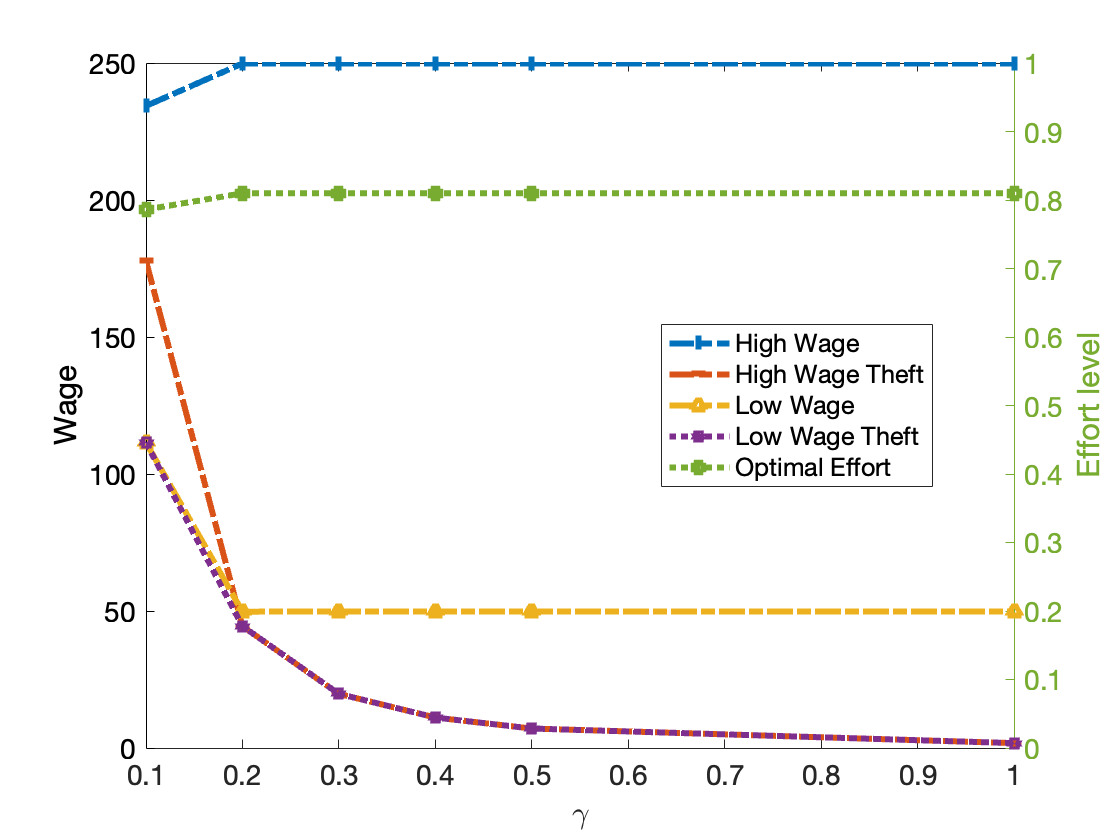}
  \caption{$\sigma=0.5$}
  \label{fig:sub31}  
\end{subfigure}%
\begin{subfigure}{0.5\textwidth}
  \centering
  \includegraphics[width=1\linewidth]{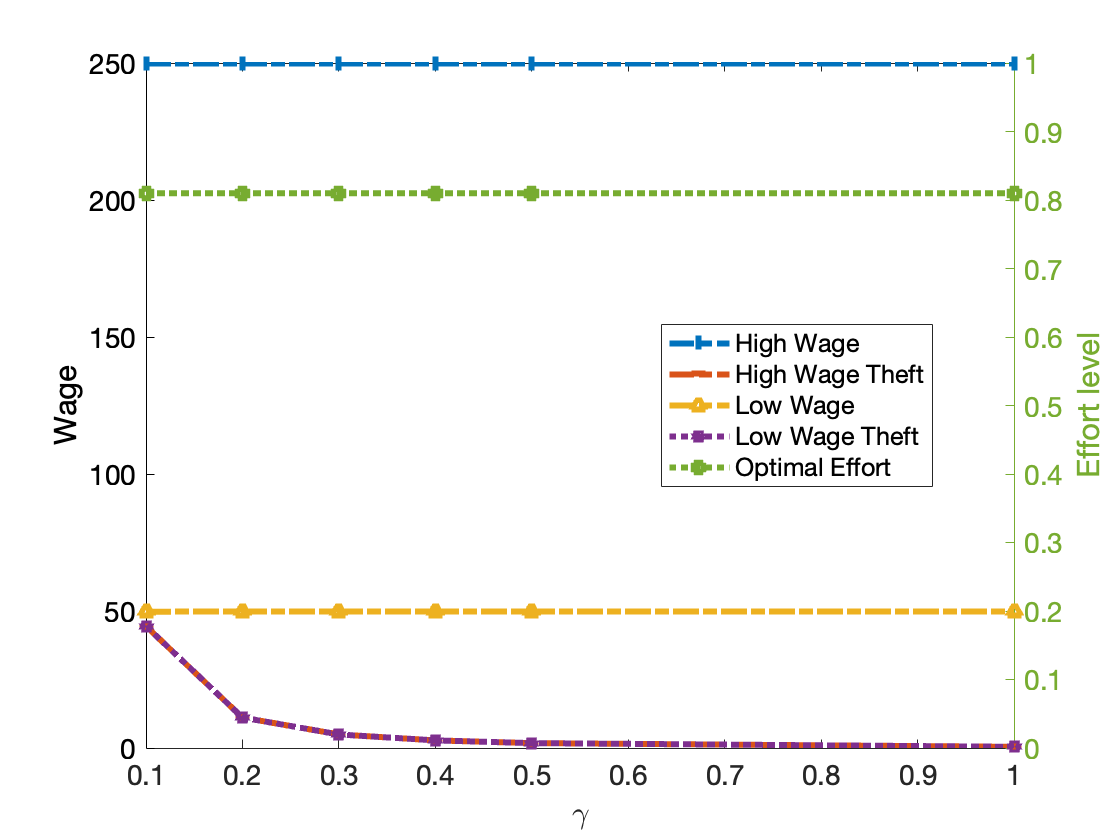}
  \caption{$\sigma=1$}
  \label{fig:sub32}
\end{subfigure}
\caption{Impact of $\gamma$ for different penalty coefficients. ($P=10$, $y_H=50$, $y_L=30$, $u=200$, $k=0.1$,  $q=3$, $p=1.5$)}
\label{fig:gammaEffect}
\end{figure}

\textbf{Impact of reservation utility on wage theft:} We now consider the impact of external utility on the verbal contract and wage theft amounts. In Figure \ref{fig:ubarEffect}, we present computational results for different reservation utilities under two different inspection rates, $\gamma=0.1$ and $\gamma=0.3$. 

\begin{figure}[ht]
\centering
\begin{subfigure}{0.5\textwidth}
  \centering
  \includegraphics[width=1\linewidth]{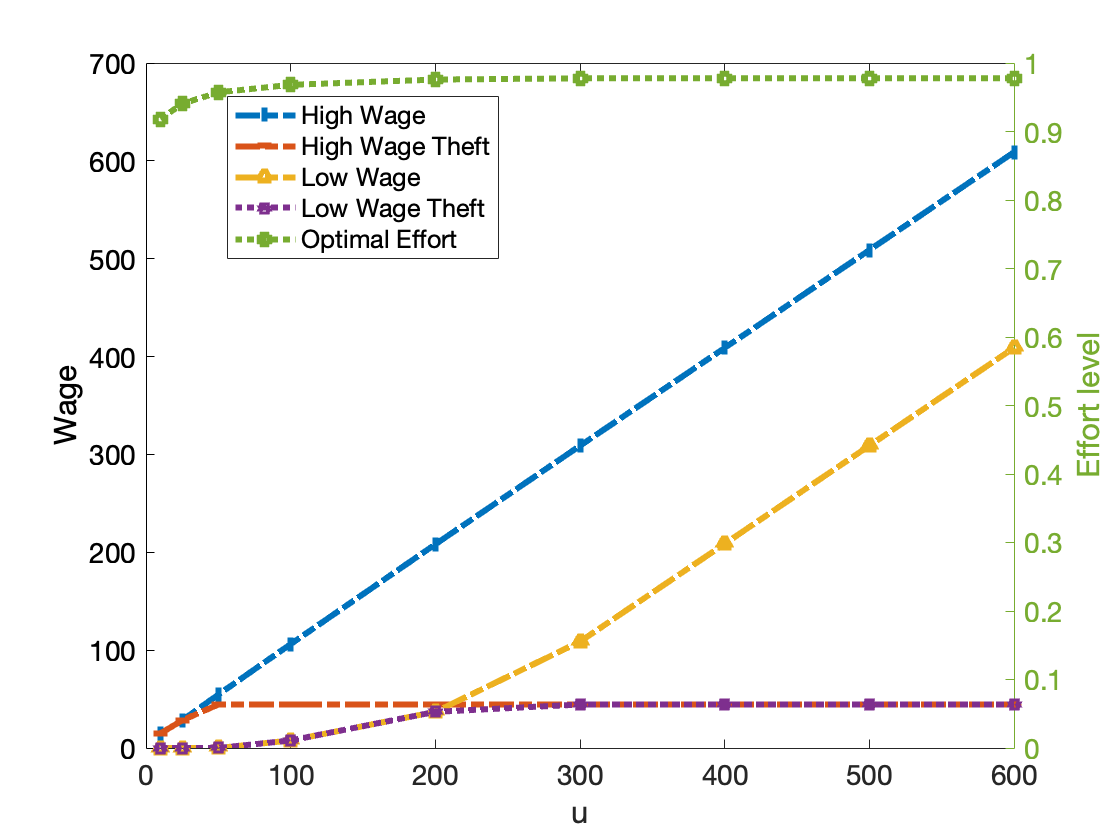}
  \caption{$\gamma=0.1$}
  \label{fig:sub41}  
\end{subfigure}%
\begin{subfigure}{0.5\textwidth}
  \centering
  \includegraphics[width=1\linewidth]{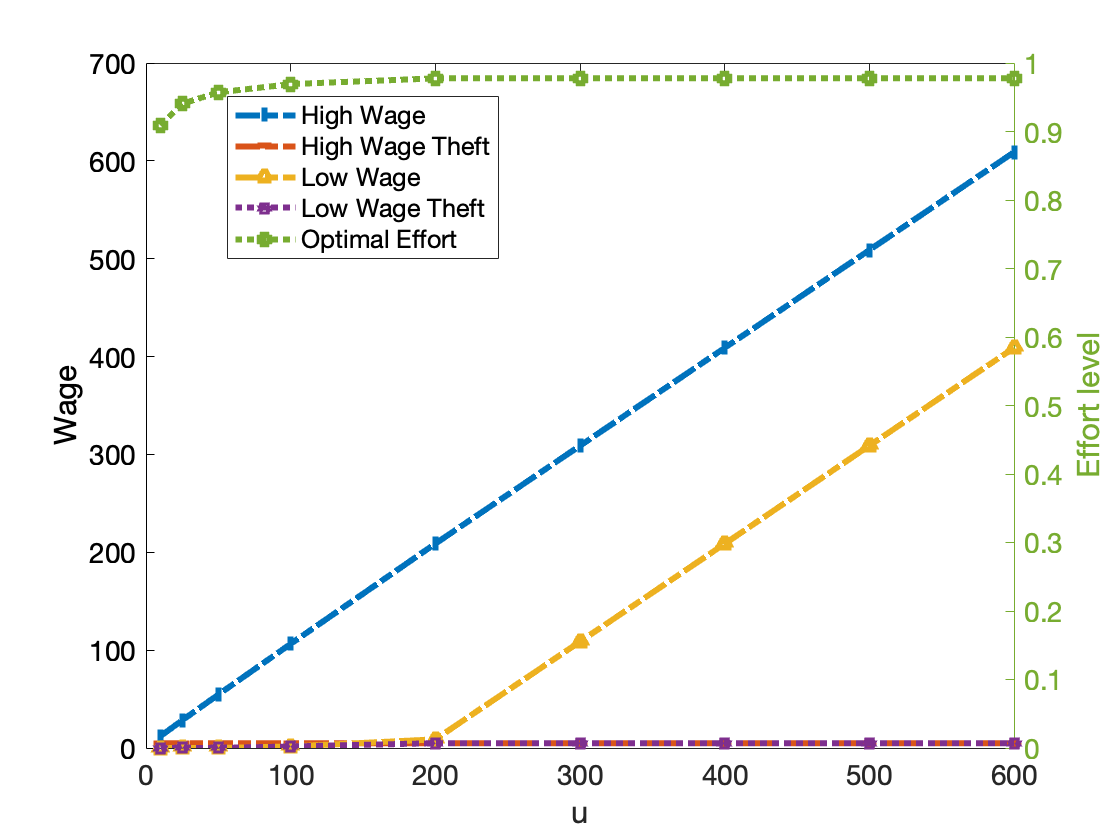}
  \caption{$\gamma=0.3$}
  \label{fig:sub42}
\end{subfigure}
\caption{Impact of $u$ for different inspection rates ($P=10$,  $y_H=50$, $y_L=30$, $\sigma=1$, $k=0.1$,  $q=1$, $p=1.5$).}
\label{fig:ubarEffect}
\end{figure}

Note that the reservation utility has no effect on the employer's ideal wage theft $\beta$. Therefore, the reservation utility affects the wage theft indirectly. 
As the reservation utility increases, the promised wages increase as well. This is expected since the employer needs to make the verbal offer more attractive than the external offers (reservation utility) due to  (\ref{Eq:IR}). 
However, this increase does not follow the same pattern on high and low verbal wages; while the high wage is increasing due to more competitive external options, the low wage does not benefit from having better external options when $u$ is less than the revenue gap between the high and low outputs (i.e., when $u \leq P(y^H-y^L)$), as observed in Figure \ref{fig:sub42} ($u \leq 200$). 
This implies, when wage theft is present, the efforts to create better external opportunities may be less effective in improving actual wages (i.e., verbal wage minus the wage theft) unless their reservation utility exceeds a certain threshold.

These results demonstrate that an increase in reservation utility does not necessarily translate to the same increase in expected effective wages when there is a low barrier to theft; while the promised wages increase, the weak deterrence against theft results in a proportion of the additional wages being stolen. 
These artificially inflated promised wages also create a perception of higher reservation utilities for workers, which could create an unfair competition for companies that comply with the law \citep{minkler2014wage}.

\textbf{Impact of worker's cost functions on wage theft:} 
Another factor that is partially under the worker's control is the cost function. 
In the cost function $C(a)=\frac{ka}{(1-a)^q}$, where $0\leq k$ and $0<q$, the cost of an effort level $a$ increases as $k$ or $q$ increase. 
Therefore, smaller values of $k$ and $q$ can be interpreted as the worker becoming more efficient or more capable. 
We remark that in some settings the reservation utility of a worker improves as the worker's ability increases, i.e., workers with more capabilities can find better jobs. 
However, this is not necessarily the case for workers in vulnerable situations who may not have an access to the same opportunities, which limits their reservation utility.
Thus, our experiments isolate the direct impact of worker skill level without considering the indirect effect caused by increased reservation utilities.

We present sample computational results associated with different cost function parameters in the figures below. Figure \ref{fig:kEffect} shows results for different $k$ values when $q \in \{1, 3\}$, and Figure \ref{fig:qEffect} displays the results for increasing $q$ values when $k\in \{0.1, 1\}$.
Recall that the worker's cost function has no impact on the employer's ideal wage theft $\beta$. Hence, similar to the reservation utility, the cost function affects the wage theft quantities indirectly.

\begin{figure}[!ht]
\centering
\begin{subfigure}{0.5\textwidth}
  \centering
  \includegraphics[width=1\linewidth]{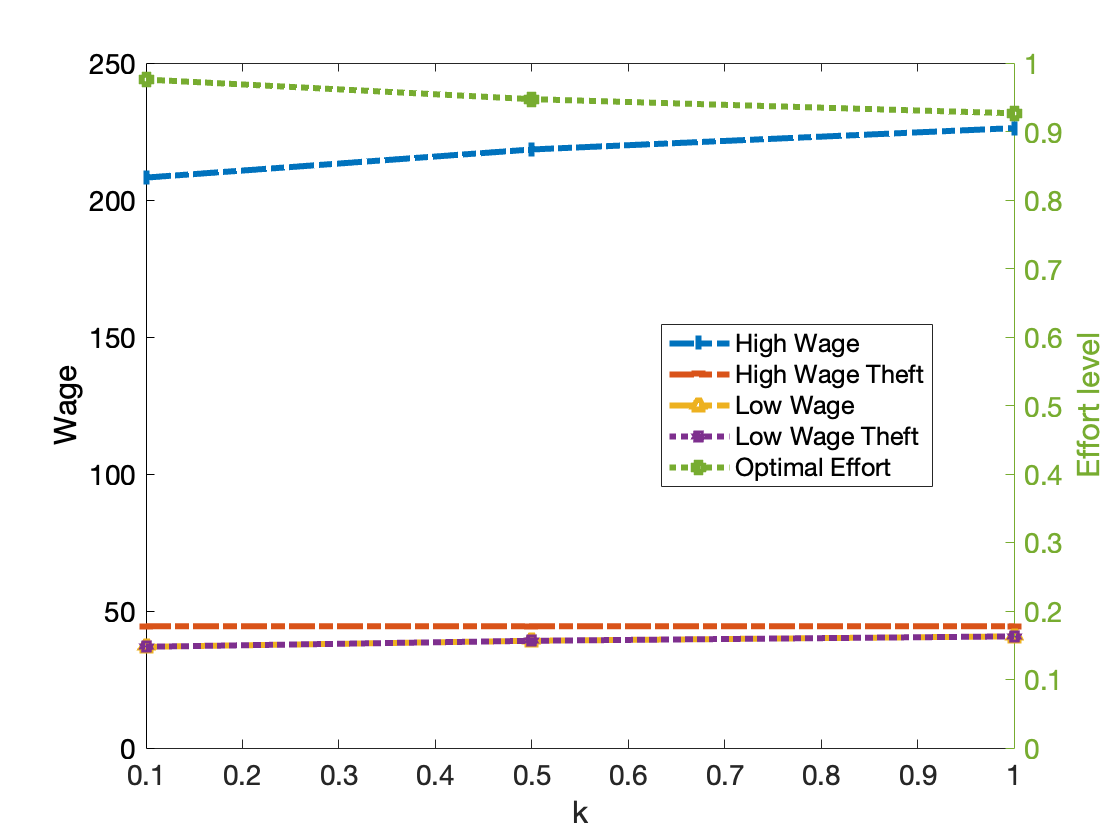}
  \caption{$q=1$}
  \label{fig:sub51}  
\end{subfigure}%
\begin{subfigure}{0.5\textwidth}
  \centering
  \includegraphics[width=1\linewidth]{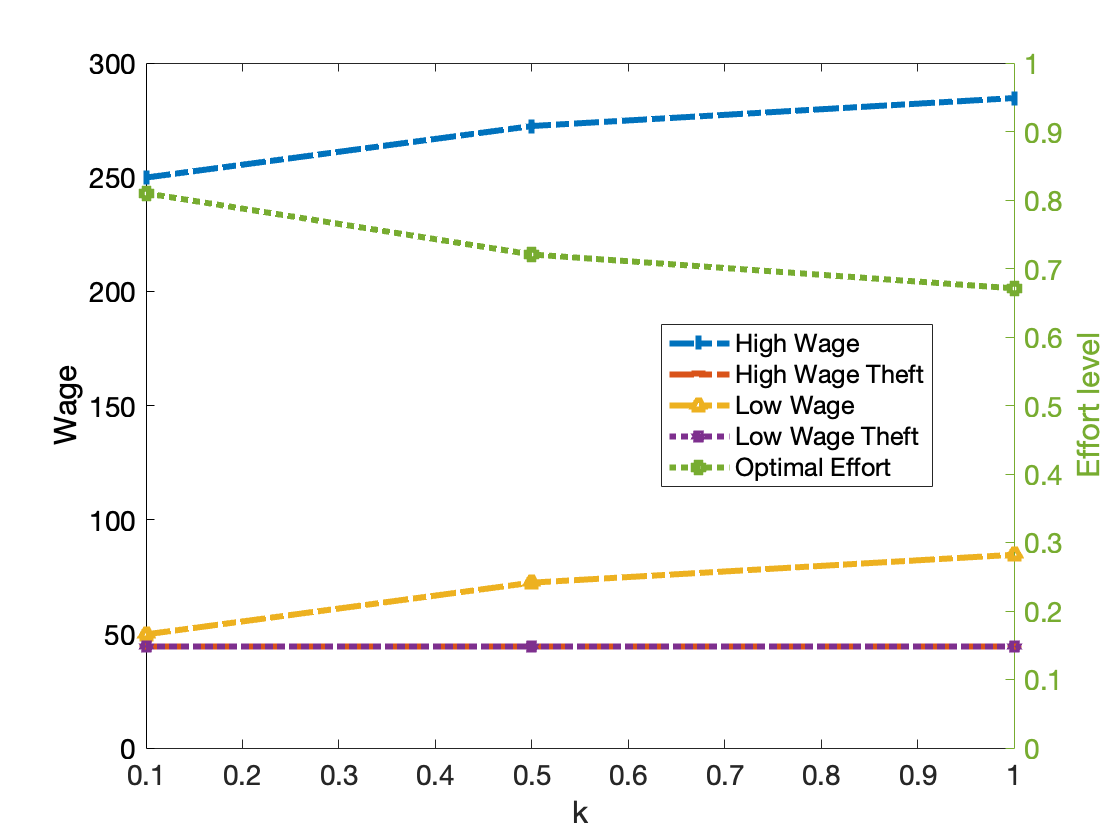}
  \caption{$q=3$}
  \label{fig:sub52}
\end{subfigure}
\caption{Impact of $k$ for different $q$ values ($P=10$,  $y_H=50$, $y_L=30$, $\sigma=1$, $\gamma=0.1$, $u=200$, $p=1.5$).}
\label{fig:kEffect}
\end{figure}

\begin{figure}[!ht]
\centering
\begin{subfigure}{0.5\textwidth}

  \centering

  \includegraphics[width=1\linewidth]{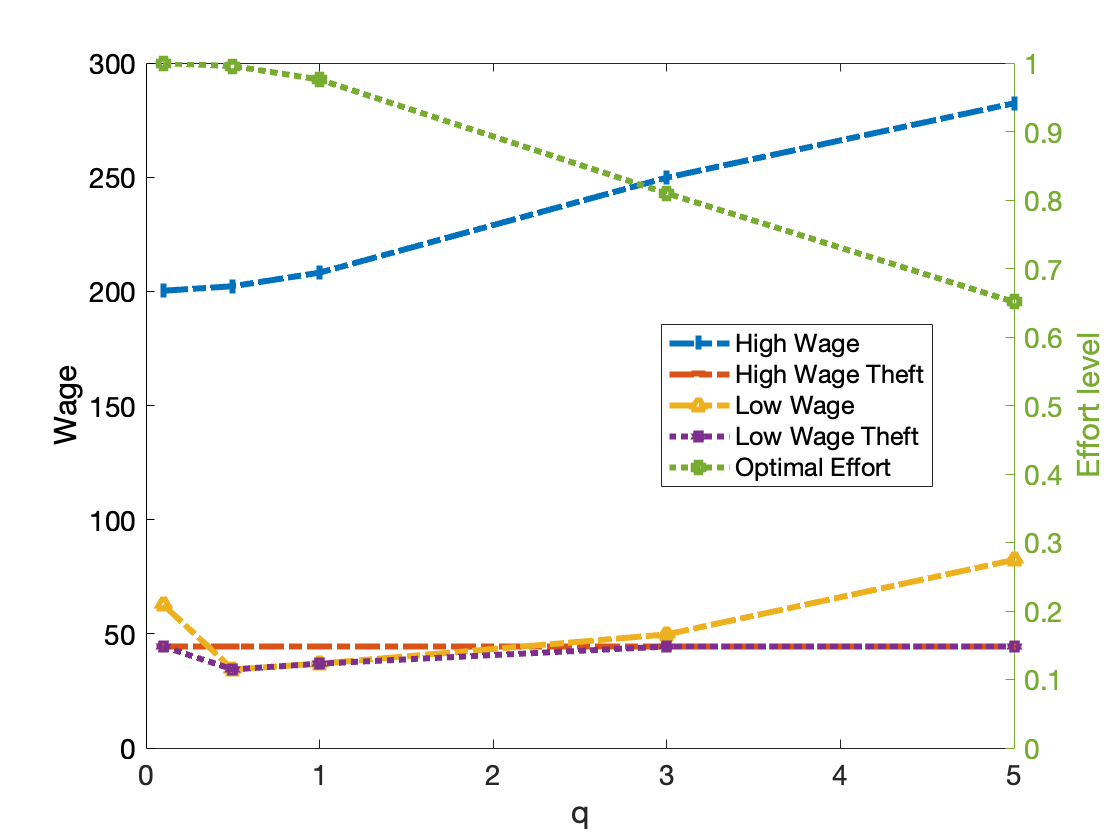}
  \caption{$k=0.1$}
  \label{fig:sub61}  
\end{subfigure}%
\begin{subfigure}{0.5\textwidth}
  \centering
  \includegraphics[width=1\linewidth]{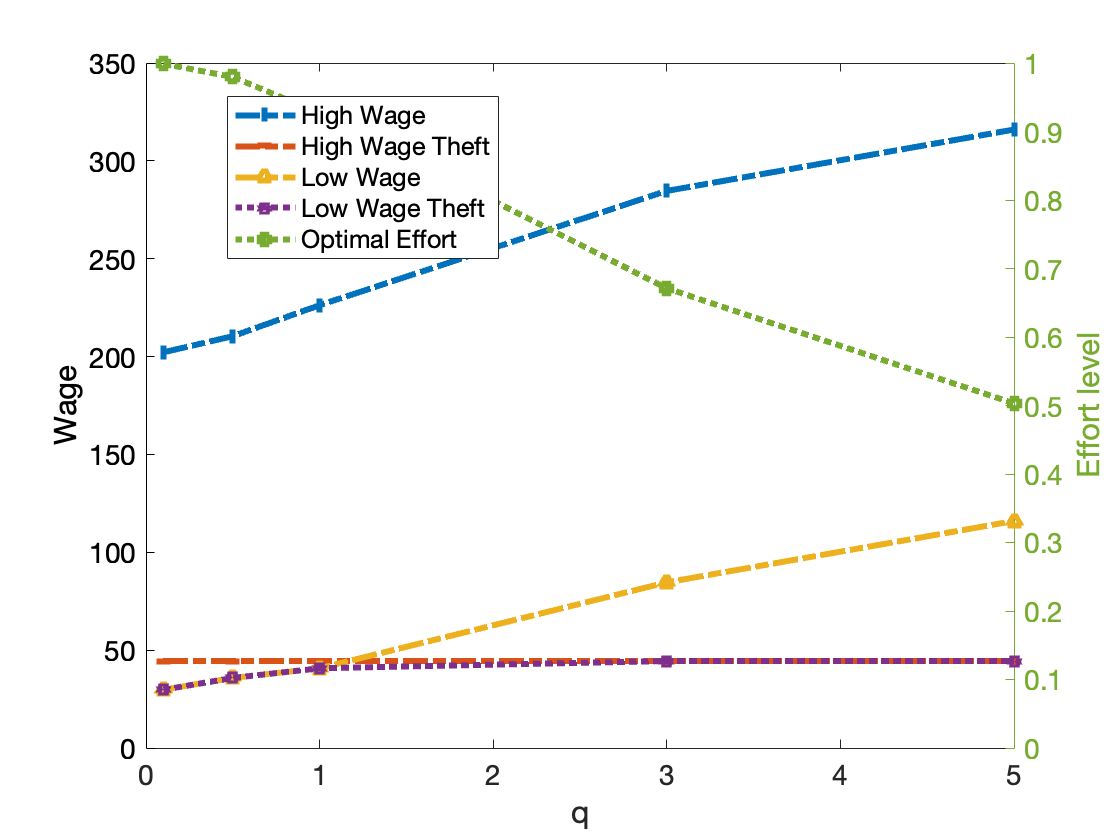}
  \caption{$k=1$}
  \label{fig:sub62}
\end{subfigure}
\caption{Impact of $q$ for different $k$ values ($P=10$,  $y_H=50$, $y_L=30$, $\sigma=1$, $\gamma=0.1$, $u=200$, $p=1.5$).}
\label{fig:qEffect}
\end{figure}

As seen in the figures, when the effort becomes more costly for the worker, the employer needs to offer a higher verbal high wage to compensate the worker accordingly and the effective wage theft also increases due to bounded ideal theft. 
Further, as the cost of effort increases, the optimal effort  level (equivalently, the probability of a high outcome) $a^{*}$ decreases, which becomes significant if the cost of effort increases exponentially (as $q$ increases) as seen in Figure \ref{fig:qEffect}. 
In response, the employer may offer a higher verbal low-wage $w_L$, which can be interpreted as a form of insurance for workers in case of low output, thereby increasing the expected promised wage and ensuring worker participation.

Overall, factors including reservation utility and the worker's cost function do not have an immediate effect on the employer's ideal wage theft, and hence do not directly change the wage theft quantity. 
Instead, if the change in these parameters leads to increased verbal offers, the workers become better off in terms of the percentage of the wage theft (i.e., $b_i/w_i$, $i\in{H, L}$) since $b_i$ does not change when $w_i> \beta$.

\section{Impact of Worker Awareness on Wage Theft}\label{Sec:HasAwareness}

As established in the previous sections, the only way in the current model to completely mitigate wage theft when the worker has no awareness is to have large penalties and inspection rates.  
As suggested by Theorem \ref{thm:eliminateTheft}, the required penalties and inspections rates can be quite large, e.g., with a large inspection rate of $\gamma=10\%$, the marginal penalty would need to be at least $1/\gamma=10$ times larger than the marginal theft, which does not reflect the typical penalty schemes currently in use \citep{Backpay}.

In this section, we explore another avenue to mitigate wage theft. 
Specifically, we consider the potential benefit of raising wage theft awareness. 
For this purpose, we introduce a multi-period, repeated game where the worker has a forecast for the wage theft in the current period based on historical occurrences of theft. 
These historical observations may be based on personal experiences, shared experiences via word-of-mouth, or advertisement campaigns targeting wage theft. 
Further, employers are aware that workers have this forecast when determining their effort level.
We then characterize the dynamic game model to understand the potential benefit of raising wage theft awareness. 

In the repeated game setting, at time period $t$, the employer determines the promised wages $w_i^t$ for output level $i\in \{H,L\}$ and the amount of wage theft $b_i^t$ for $i\in \{H,L\}$, as before. 
The worker then determines their effort level $a^t$ based on the promised wages $w_i^t$ for $i\in \{H,L\}$. 
However, in contrast to the one-stage model, the worker has a forecast of the  wage theft $\hat{b}_i^t$ as a function of $(b_i^1,...,b_i^{t-1})$ for $i\in \{H, L\}$ based on historical observations of wage theft.
As a result, a worker believes they will only receive $(w_i^t-\hat{b}_i^t)^+=\max\{0, w_i^t-\hat{b}_i^t\}$ for outcome $i\in \{H,L\}$.
This results in the following repeated-game.
\begin{align}
    \max_{\a^t, w_H^t, w_L^t, b_H^t, b_L^t} &\ \a^t (\H - w_H^t +b_H^t - \P(b_H^t))+(1-\a^t) (\L - w_L^t +b_L^t - \P(b_L^t)) \tag{Employer Profit}\\
  \a^t &\in \arg\max_{a'\in [0,1)} \left\{ \a' (w_H^t{\color{black}-\hat{b}^{t}_H})^+ + (1-\a')(w_L^t{\color{black}-\hat{b}_L^t})^+ - \C(\a') \right\}\tag{Incentive Compatibility}\\
   \a^t& (w_H^t{\color{black}-\hat{b}^{t}_H})^+ + (1-\a^t) (w_L^t{\color{black}-\hat{b}_L^t})^+ - \C(a^t) \geq u \tag{Individual Rationality}\\  
    0&\leq b_L^t \leq w_L^t, \ \
    0\leq b_H^t \leq w_H^t,\tag{Wage Bounds}\\
    \a^t& \in [0,1] \tag{Probability of High Outcome}
\end{align}

There are many methods for a worker to forecast future wage theft.  
We introduce the class of \textit{time-converging} forecasts.
A time-converging forecast simply will eventually predict the actual wage theft if the wage theft becomes stationary. 
Formally, a forecast $\hat{b}$ is time-converging if $b_i^{t}=b_i$ for all $t\geq t_0$ implies $\hat{b}_i^t(b^1_i,...,b^{t_0}_i,b_i,..., b_i)\to b_i$ as $t\to \infty$.
Examples of common time-converging forecasts include the most recent theft $\hat{b}_i^t=b_{i}^{t-1}$, the average theft $\hat{b}_i^t=\sum_{s=1}^t b_i^s/t$, and standard forecasting rules such as moving averages and exponential smoothing rules.
We assume that workers use a time-converging forecast, which is a weak assumption given standard methods are time-converging.

Similarly, there are many methods for the employer to determine their strategy in each iteration which may include methods built around bounded rationality, online optimization techniques, reinforcement learning, and subgame perfect equilibria.
In this paper, we focus on an employer that uses a fixed strategy, namely, we assume the employer offers the same wages and plans to steal the same amount in every iteration, i.e., the employer employs a fixed strategy of $(w_L^{t},w_H^{t},b_L^{t},b_H^{t})=(w_L,w_H,b_L,b_H)$. This assumption of fixed wages/theft is well supported in practice.  Studies show that small businesses are less likely to invest in sophisticated payroll systems and prefer more straightforward and consistent wage practices \citep{NFIB2017}. In areas with low regulatory scrutiny, \cite{Bernhardt2009} showed that 
employers are more likely to engage in wage theft without frequent adjustments to their strategies. In addition, many companies establish fixed wage policies to ensure uniformity and reduce administrative burden (e.g., piece-rate payment systems in the garment industry). These practices can become habitual, especially in businesses with high employee turnover, where making continuous strategy adjustments is less practical \citep{CooperKroeger2017}.

Assuming the employer uses a fixed strategy, and since the workers use a time-converging forecast, the forecast for the wage theft $\hat{b}_i^t$ also converges to $b_i^*$ thereby implying the worker effort level $a^{t*}$ also converges to an invariant $a^*$. 
As a result, the repeated-game reduces to the following one-stage game:
\begin{align}
    \max_{\a, w_H, w_L, b_H, b_L} &\ \a\cdot (\H - w_H +b_H - \P(b_H))+(1-\a)\cdot (\L - w_L +b_L - \P(b_L)) \tag{Employer Profit}\\
  \a &\in \arg\max_{a'\in [0,1)} \left\{ \a'\cdot (w_H-b_H)^+ + (1-\a')\cdot (w_L-b_L) - \C(\a') \right\} \tag{Incentive Compatibility}\\
   \a&\cdot (w_H-b_H) + (1-\a)\cdot (w_L-b_L)^+ - \C(a) \geq u \tag{Individual Rationality}\\  
    0&\leq b_L \leq w_L, \ \
    0\leq b_H \leq w_H, \tag{Wage Bounds}\\
    \a& \in [0,1] \tag{Probability of High Outcome}
\end{align}
In Theorem \ref{thm:optimalfixed}, we show worker awareness can eliminate wage theft when employers use a optimal fixed strategy.

\begin{restatable}{theorem}{optimalfixed}\label{thm:optimalfixed}
      Suppose $\c(0)<\H-\L$ and that workers have a time-converging forecast of wage theft in the repeated wage theft game where the employer and worker use fixed strategies. 
    Then there exists an optimal fixed strategy $(a^*, w_H^*, w_L^*, b_H^*, b_L^*)$ where $a^*\in (0,1)$ and $b_H^*=b_L^*=0$.   
\end{restatable}

Theorem \ref{thm:optimalfixed} indicates that in a repeated wage theft game where both the employer and worker follow fixed strategies and workers have a time-converging forecast of wage theft, an optimal fixed strategy exists. In this optimal strategy, the worker exerts a positive level of effort $a \in (0,1)$ and the employer does not engage in wage theft $b^*_H=b^*_L=0$. This result highlights that, in the long run, avoiding wage theft is the most effective approach for employers, as honesty leads to the best outcomes when workers can predict wage theft over time. In other words, being honest is not only ethically sound but also optimal from a strategic perspective.

In practice, employers may deviate from optimal fixed strategies due to model inaccuracies, lacking full knowledge of the workers’ abilities or the actual costs of wage theft, and many other factors. However, our next result demonstrates that even when employers follow a given fixed wage strategy (not necessarily optimal), as long as workers can accurately predict potential wage theft based on historical data, it remains advantageous for employers to maintain honesty and refrain from wage theft.

\begin{restatable}{theorem}{dominate}\label{thm:notheft.refined}
    Suppose that workers have a time-converging forecast of wage theft in the repeated wage theft game where the employer and worker use fixed strategies. 
    Then the employer strategy of promising $w_i$ and stealing $b_i$ for $i\in \{H,L\}$ is (weakly) dominated by the strategy of promising $(w_i-b_i)$ and stealing $0$ for $i\in \{H,L\}$.
    The dominance is strict when $a>0$.
\end{restatable}

The proofs of both theorems appear in Appendix \ref{app:Repeated}.
In summary, Theorems \ref{thm:optimalfixed} and \ref{thm:notheft.refined} indicate that when workers can effectively forecast wage theft over time, the employer's strategy of promising a higher wage and then stealing part of it is no better than simply being honest and offering a lower wage upfront with no theft. Being honest proves to be the superior long-term strategy for the firm, as workers will eventually adjust their behavior in response to wage theft. This underscores the value of enhancing worker awareness as a mechanism to mitigate wage theft. Platforms that disclose historical wage theft data for specific employers could serve as powerful tools in combating wage theft by promoting transparency, honesty, and accountability.

\section{Conclusion}
\label{Sec:Conclusions}
In this paper, we analyze the dynamics of wage theft in day labor markets using both a single-period and a repeated principal-agent model, where employers may withhold wages from workers based on the outcome of short-term jobs. We explore scenarios where workers are either unaware or aware of potential wage theft. Our analysis reveals that wage theft is likely to persist even with frequent inspections unless penalties are substantially increased. We also find that workers with lower reservation utilities are disproportionately affected, losing a larger portion of their wages, and that improving worker efficiency does little to mitigate wage theft without better employment opportunities. Additionally, we consider the impact of worker awareness on wage theft, showing that informed workers can reduce the likelihood of exploitation. This underscores the importance of not only enforcing stronger penalties but also enhancing worker awareness through education and community efforts. Our findings contribute to a deeper understanding of wage theft dynamics and suggest pathways for future research aimed at developing more effective strategies to protect vulnerable workers in these informal labor markets.

As this work represents only a first step in the theoretical examination of wage theft, the future research directions are numerous. For example, future work could explore the heterogeneity of employers, examining how differences in risk aversion, resources, and regulatory scrutiny influence wage theft strategies. Another promising direction involves the role of technology, investigating how innovations like blockchain for wage tracking or mobile reporting systems could improve transparency and reduce wage theft. Furthermore, studying the psychological and behavioral aspects of wage theft, such as its impact on worker morale or employers' ethical decision-making, could yield new insights. Finally, comparative studies across sectors or countries with varying regulatory frameworks could help identify institutional factors that either exacerbate or mitigate wage theft.

\bibliographystyle{informs2014} 
\bibliography{reference.bib}

\begin{appendices}

\section{Proof of Proposition \ref{prop:IdealTheft}: Optimal Wage Theft}\label{app:theft}
\IdealTheft*
\begin{proof}
    We show the result only for $b_H$; the result follows identically for $b_L$. 
    By definition, $b_H\leq w_H$.
    Recall the objective value for the wage theft problem is
    \begin{align*}
        g(\a, w_h, w_l, b_h, b_l):=&\ \a\cdot (\H - w_H +b_H - \P(b_H))\ +\\ &(1-\a)\cdot (\L - w_L +b_L - \P(b_L)).
    \end{align*}
    Consider any feasible solution where $b_H< \min\{\beta, w_H^*\}$.
    First, suppose the feasible solution has $\a=0$. 
    Since the only constraint containing $b_H$ is $b_H\leq w_H$, we may increase $b_H$ to $b_H'=\min\{\beta, w_H^*\}$ while maintaining feasibility. 
    Further, since $\a=0$, $b_H$ has no impact on the the objective value and therefore the new feasible solution has the same value. 
    Thus, if there is an optimal solution with $\a^*=0$, then there is at least one optimal solution with $b_H^*=\min\{\beta, w_H^*\}$.

    Second, suppose the feasible solution has $a>0$.     
    Then $dg/db_H= \a\cdot (1 - \P'(b_H))> 0$, since $b_H< \min\{\beta,w_H^*\}$, $\P'(\beta)=1$, and $\P$ is strictly convex (implying $\P'$ is strictly increasing). 
    Further, since the only constraint containing $b_H$ is $b_H\leq w_H$, we may increase $b_H$ thereby improving the quality of the current solution. 
    Therefore the optimal solution must have $b_H^*=\min\{\beta,w_H^*\}$.  

    The proof for $b_L$ follows identically with cases $a=1$ and when $a<1$. 
\end{proof}

\section{Proof of Proposition \ref{prop:Wages}: Optimal Wages}

\label{app:wages}

\Wages*

\begin{proof}
Our proof is in two parts; we first prove the result on optimal high wage and then on the optimal low wage.

The (\ref{Eq:IC}) constraint is 
\begin{align*}
    a\in \arg\max_{a'\in [0,1]} IC(a'):=a'\cdot w_H + (1-a')\cdot w_L-C(a'). 
\end{align*}

Observe that $IC'(a)=w_H-w_L-C'(a)$ and $IC'(0)\geq 0$ and $IC'(1)\to -\infty$ and therefore for $a$ to maximize $IC(a)$, $IC'(a)=0$ and $w_H=w_L+C'(a)$, which completes the first part of the proof.
Further, $IC'$ is strictly increasing with respect to $w_H$ and therefore $w_H=w_L+C'(a)$ is the unique solution for $IC'(a)=0$. 

By Proposition \ref{prop:IdealTheft}, $b_i=\min\{\beta, w_i\}$. 
By the first part of the proof, we have $w_H=w_L+\c (\a)\geq w_L \geq 0$.  
After making these substitutions, rearranging terms, and removing the now-redundant constraint $w_H\geq 0$, the optimization problem reduces to:
\begin{align*}
\max_{\a, w_L} \ h(\a, w_L)= \a\cdot &(\H - w_L-\c (\a) +\min\{\beta, w_L+\c (\a)\} - \P(\min\{\beta, w_L+\c (\a)\})) \\
   +(1-\a)\cdot &(\L - w_L +\min\{\beta, w_L\} - \P(\min\{\beta, w_L\}))\tag{Employer Profit}\\ 
  w_L &+ \a \cdot \c (\a) - \C(a)\geq u \tag{Individual Rationality}\\
    w_L&\geq 0 \tag{Non-Negative Wages}\\
    a& \in [0,1] \tag{Probability of High Outcome}
\end{align*}

Observe that the objective function $\ h(\a, w_L)$ is a function of $w_L$ for any fixed effort level $\a$. 
For a fixed $\a$, we first break the objective into three cases: $w_L \in [0, \beta - \c (\a)]$ (equivalently, $w_L\leq w_H=w_L+\c(\a)\leq \beta$) , $w_L \in [\beta - \c (\a), \beta]$ ($w_L\leq \beta$ but $w_H\geq \beta$), and $w_L \in [\beta, \infty)$ ($w_H\geq w_L \geq \beta$), and show that $\partial h / \partial w_L\leq 0$ in each region.

Case 1: $w_L \in [0, \beta - \c (\a)]$, implying $w_L \leq \beta-\c(a)$ and  $w_L \leq w_H \leq \beta$. 
Therefore $b_i=w_i$ and the objective function simplifies to:
\begin{align*}
    h_1(\a, w_L):&= \a\cdot (\H - \P(w_L+ \c (\a)) + (1-\a)\cdot (\L - \P(w_L)),\\
    \frac{\partial h_1}{\partial w_L}&= -\a\cdot \P'(w_L + \c (\a)) - (1-\a) \cdot \P'(w_L)  \leq 0, 
\end{align*}
since $\eta$ is convex and $\eta'(0)\geq 0$ indicating $\eta'(w_L+\c (\a))\geq \eta'(w_L) \geq 0$.
Thus $\partial h / \partial w_L\leq 0$, completing case 1.

Case 2: $w_L \in [\beta - \c (\a),\beta]$ implying $w_L\leq \beta \leq w_H$. 
Therefore $b_L=w_L$ and $b_H=\beta$ and the objective function reduces to:
\begin{align*}
    h_3(\a, w_L):&= \a\cdot (\H - w_L - \c (\a) + \beta - \P(\beta)) + (1-\a)\cdot (\L - \P(w_L)),\\
    \frac{\partial h_2}{\partial w_L}&=-\a -(1-\a)\cdot \P'(w_L)\leq 0
\end{align*}
thereby completing case 2. 

Case 3: $w_L \in [\beta,\infty)$ implying $b_i=\beta$. The objective function is:
\begin{align*}
    h_3(\a, w_L):&= \a\cdot (\H - w_L - \c (\a) + \beta - \P(\beta)) + (1-\a)\cdot (\L - w_L + \beta - \P(\beta)),\\
    \frac{\partial h_3}{\partial w_L}&=-1 \leq 0. 
\end{align*}
thereby completing all three cases and $\partial h / \partial w_L\leq 0$. 
Hence, since $h$ is continuous, it is optimal to make $w_L^*$ as small as possible without changing the value of any other variables. The feasibility constraints of the problem require that $w_L \geq u-\a\cdot \c (\a)+\C (\a)$ and  $w_L \geq 0$, indicating $w_L^*=\max\{0, u-\a\cdot \c (\a)+\C (\a)\}$.  

To establish uniqueness of $w_L^*$ when $\a\in (0,1)$, simply observe that all three cases result in $\partial h/\partial w_L< 0$ when $\a  \in (0,1)$ and when $w_L>0$. 
\end{proof}

\section{Proof of Proposition \ref{prop:existence}}\label{app:exist}

\Existence*

\begin{proof}
    The one-dimensional search problem is defined over continuous functions and, if the domain were compact, then the Weierstrass extreme value theorem would imply the existence of an optimal solution.  
    However, the function domain is $\a\in [0,1)$, an open space, and therefore the Weirstrass theorem does not immediately imply.  
    We instead show there exists an $\a'<1$ where the objective function is strictly decreasing for all $\a>\a'$, which will imply that it suffices to optimize over the compact domain $[0,\a']$ and guarantee the existence of an optimal solution.

    We first show that $\c(a)\to \infty$ as $a\to 1$. Since $\c$ is increasing ($\C$ is convex), it suffices to show $\c$ is unbounded.  
    Suppose for contradiction that $\c(\a)\leq w$, then $\lim_{a\to \infty} \C(\a) = \c(0) + \int_0^1 \c(a)da \leq \c(0) + \int_0^1 wda= \c(0) + w$ contradicting that $\C(\a)\to \infty$.

    Next, we analyze the function $w_L^*(\a)=\max\{0,u-a\cdot \c(\a) + \C(a)\}$ from Proposition \ref{prop:Wages} for $\a\geq 1/2$. 
    \begin{align*}
        \a\c(\a) - \C(a) & = \int_0^\a \alpha \c'(\alpha)d\alpha \\
        & = \int_0^{1/2} \alpha \c'(\alpha)d\alpha + \int_{1/2}^\a \alpha \c'(\alpha)d\alpha \\
        & \geq \int_0^{1/2} \alpha \c'(\alpha)d\alpha + \int_{1/2}^\a 
        \frac{1}{2}\c'(\alpha)d\alpha \tag{$\C$ is convex $\Rightarrow \c'\geq 0$}\\
        &= \frac{1}{2}\c\left(\frac{1}{2}\right)-\C\left(\frac{1}{2}\right) + \frac{1}{2}\c(\a) - \frac{1}{2}\c\left(\frac{1}{2}\right)\\
        &= \frac{1}{2}\c(\a) - \C\left(\frac{1}{2}\right) \to \infty \ as \ a\to 1.
    \end{align*}
    Thus, there is an $\a''_1<1$ such that $\a\c(\a) - \C(a)\geq u$ for all $\a > \a''_1$, implying $w_L^*(\a)=b_L^*(\a)=0$ for $\a > \a''_1$.  
    Further, we can select $\a''_2$ such that $\c(\a''_2) > \beta$ since $\c(\a)\to \infty$, hence $\c(\a)>\beta$ for all $\a \geq \a''_2$. Pick $\a'' = \max \{\a''_1, \a''_2\}$, 
    by Proposition \ref{prop:Wages}, $w_H^*(\a)=\c(\a)$ and $b_H^*(\a)=\beta$ for all $a> \a''$. 

    Recall the (\ref{eqn:1DSearch}) maximizes
\begin{align*}
    g(\a):&= \a\cdot \left( \H - w_H^*(\a) + b_H^*(\a) - \P(b_H^*(\a))\right) + (1-\a)\cdot \left( \L - w_L^*(\a) + b_L^*(\a) - \P(b_L^*(\a))\right)
\end{align*}
The derivative, when well-defined, is given by: 
\begin{align*}
     g'(\a) = \ & \left( \H - w_H^*(\a) + b_H^*(\a) - \P(b_H^*(\a))\right) - \a \cdot \left( \frac{dw_H^*}{d\a} - \frac{db_H^*}{d\a} + \frac{db_H^*}{d\a}\cdot \P'(b_H^*(\a))\right)\\
     & - \left( \L - w_L^*(\a) + b_L^*(\a) - \P(b_L^*(\a))\right)- (1-a)\cdot \left( \frac{dw_L^*}{d\a} - \frac{db_L^*}{d\a} + \frac{db_L^*}{d\a}\cdot \P'(b_L^*(\a))\right) 
\end{align*}
    While $g$ is continuous, $g'$ is discontinuous at the break points of $w_L$, $b_L$ and $b_H$. 
    However, by selection of $\a''$, $w_L(\a'')=b_L(\a'')=0$ and $b_H(\a'')=\beta$ and no additional break points occur when $\a>\a''$. 
    Thus, for $\a> \a''$, the derivative is well-defined and simplifies to 
    \begin{align*}
            g'(\a) = \ &  \H - w_H^*(\a) + \beta - \P(\beta) - \a \cdot  \frac{dw_H^*}{d\a}  -  \L  \\
                & = \H -  \L + \beta - \P(\beta) - \c(\a)  - \a \cdot  \c'(\a) \\
                &\to -\infty \ as \ a\to 1
    \end{align*}

    Thus, there is an $\a'\in (\a'', 1)$ such that $g(\a)$ is strictly decreasing for all $\a> \a'$.
    Therefore, it suffices to solve the (\ref{eqn:1DSearch}) over the compact domain $[0,a']$, and by the Weierstrass extreme value theorem, there exists an optimal solution to the (\ref{eqn:1DSearch}) with $a^*\in [0,\a']\subset [0,1)$. 
\end{proof}

\section{Proof of Proposition \ref{prop:positive}}\label{app:positive}
\Positive*

\begin{proof}
 We show the objective function of the (One-dimensional search problem) is such that $g'(0)>0$ implying $\a=0$ is sub-optimal. 
The derivative of the objective function is 
\begin{align*}
     g'(\a) = \ & \left( \H - w_H^*(\a) + b_H^*(\a) - \P(b_H^*(\a))\right) - \a \cdot \left( \frac{dw_H^*}{d\a} - \frac{db_H^*}{d\a} + \frac{db_H^*}{d\a}\cdot \P'(b_H^*(\a))\right)\\
     & - \left( \L - w_L^*(\a) + b_L^*(\a) - \P(b_L^*(\a))\right)- (1-a)\cdot \left( \frac{dw_L^*}{d\a} - \frac{db_L^*}{d\a} + \frac{db_L^*}{d\a}\cdot \P'(b_L^*(\a))\right) 
\end{align*}

By Proposition \ref{prop:Wages}, for sufficiently small $\a$, $w_L^*(\a)=\max\{0, u-\a\c(\a)+\C(\a)\}=u-\a\c(\a)+\C(\a)$.
Therefore $\frac{dw_L^*}{d\a}=-a\c'(\a)$ which is 0 when $\a=0$. 
 Note that $b_L^*(\a)= \min\{\beta, w_L^*(\a)\}$. For sufficiently small $\a$, if $\beta \leq w_L^*(a)$, then $b_L^*(\a)= \beta$ and $\frac{db_L^*}{d\a}|_{a=0}=0$. If $\beta > w_L^*(a)$, then $b_L^*(\a)= w_L^*(a)$ and hence $\frac{db_L^*}{d\a}|_{a=0}=\frac{dw_L^*}{d\a}|_{a=0}=0$. In both cases, we have $\frac{db_L^*}{d\a}|_{a=0}=0$.

Finally, $b_H^*=\min\{\beta, w_H^*\} \geq \min\{\beta, w_L^*\}=b_L^*$ and, by definition of $\beta$, $(b-\P(b))$ is non-decreasing for $b\leq \beta$ implying $(b_H^*-\P(b_H^*))-(b_L^*-\P(b_L^*))\geq 0$.
Therefore at $\a=0$, the derivative simplifies to:
\begin{align*}
    g'(0)&= \H - \L -(w_H^*(0)-w_L^*(0)) + (b_H^*(0)-\P(b_H^*(0)))-(b_L^*(0)-\P(b_L^*(0)))\\
        &= \H-\L -\c(0) + (b_H^*(0)-\P(b_H^*(0)))-(b_L^*(0)-\P(b_L^*(0)))\tag{By Proposition \ref{prop:Wages}}\\
        &\geq \H-\L -\c(0)>0
\end{align*}
when $\c(0)< \H-\L$, establishing the sufficient conditions. 

The necessary condition follows by the fact that $w_L^*(0)=u$,  $w_H^*(0)=u+\c(0)$. 
When $\beta\leq u$,  $b_H^*(0)=b_L^*(0)=\beta$ and the final line of $g'(0)$ holds with equality. 
\end{proof}

\section{Proofs of Theorems \ref{thm:optimalfixed} and \ref{thm:notheft.refined}}\label{app:Repeated}

We begin with Theorem \ref{thm:notheft.refined}. 

\dominate*

\begin{proof}
    Recall that when agents used fixed strategies with time-converging forecasts, that the repeated game reduces to the following one-stage game:
    \begin{align}
        \max_{\a, w_H, w_L, b_H, b_L} &\ \a\cdot (\H - w_H +b_H - \P(b_H))+(1-\a)\cdot (\L - w_L +b_L - \P(b_L)) \tag{Employer Profit}\\
      \a &\in \arg\max_{a'\in [0,1)} \left\{ \a'\cdot (w_H-b_H)^+ + (1-\a')\cdot (w_L-b_L) - \C(\a') \right\} \tag{Incentive Compatibility}\\
       \a&\cdot (w_H-b_H) + (1-\a)\cdot (w_L-b_L)^+ - \C(a) \geq u \tag{Individual Rationality}\\  
        0&\leq b_L \leq w_L, \ \
        0\leq b_H \leq w_H, \tag{Wage Bounds}\\
        \a& \in [0,1] \tag{Probability of High Outcome}
    \end{align}

    Consider an arbitrary feasible solution $(a,w_H,w_L,b_H,b_L)$ where $b_i>0$ for some $i\in \{H,L\}$. 
    As in the theorem statement, consider replacing $w_i$ with $w_i-b_i$ and $b_i$ with $b_i-b_i=0$ for each $i\in \{H,L\}$. 
    Then the updated solution remains feasible; the left-hand side values of the incentive compatibility and individual rationality constraints remain unchanged since the decrease in $w_i$ is offset by the equivalent decrease in $b_i$. 
    Furthermore, the objective value improves by exactly $\a\cdot \P(b_H)+(1-\a)\cdot \P (b_L) \geq 0$.
    Thus, the strategy $(a,w_H,w_L,b_H,b_L)$ is weakly dominated by $(a,w_H-b_H,w_L-b_L,0,0)$ and strict dominance occurs when $\P(b_H)+(1-\a)\cdot \P (b_L) >0$, e.g., when $a\in (0,1)$ and $b_i>0$.  
\end{proof}

Theorem \ref{thm:optimalfixed} then follows from Theorem \ref{thm:notheft.refined} and Proposition \ref{prop:positive}. 
Specifically, Theorem \ref{thm:optimalfixed} implies $b_i^*=0$ and the remainder of the proof follows identically to Proposition \ref{prop:positive} since for $b_i=0$, the repeated game with time-converging forecasts is equivalent to the original one-stage game. 

\optimalfixed*

\end{appendices}
\end{document}